\documentclass[a4paper,11pt,english]{article}
\pdfoutput=1

\usepackage[utf8]{inputenc}
\usepackage[english]{babel}
\usepackage[math]{kurier}
\usepackage[T1]{fontenc}
\usepackage{fullpage}

\usepackage[table]{xcolor}

\usepackage{verbatim}\long\def\ignore#1{}
\usepackage{ifdraft}
\ifdraft{\newcommand{\authnote}[3]{{\color{#3} {\bf  #1:} #2}}}{\newcommand{\authnote}[3]{}}
\newcommand{\anote}[1]{\authnote{ András}{#1}{green}}
\newcommand{\tnote}[1]{\authnote{ Tom}{#1}{red}}

\usepackage{ifthen}
\usepackage{amssymb}
\usepackage{mathtools}
\usepackage{dsfont}
\usepackage{bbm}
\usepackage{mleftright}\mleftright
\usepackage{thmtools} 
\usepackage{thm-restate}

\usepackage{enumitem}
\usepackage{float}

\usepackage{diagbox}
\usepackage{graphicx}
\usepackage{wrapfig}
\usepackage{subcaption}
\usepackage{caption}
\graphicspath{{./}{./tex/}}

\usepackage[normalem]{ulem} 

\newtheorem{theorem}{Theorem}
\newtheorem{corollary}[theorem]{Corollary}
\newtheorem{lemma}[theorem]{Lemma}

\newtheorem{definition}[theorem]{Definition}
\newtheorem{claim}[theorem]{Claim}

\newtheorem{conjecture}[theorem]{Conjecture}

\newenvironment{proof}
{\noindent {\bf Proof. }}
{{\hfill $\Box$}\\	\smallskip}

\usepackage[final]{hyperref}
\hypersetup{
	colorlinks = true,
	allcolors = {blue},
}

\renewcommand{\P}{\mathbb{P}}
\newcommand{\R}{\mathbb{R}}
\newcommand{\E}{\mathbb{E}}
\newcommand{\II}[1]{\mathrm{II}^{(#1)}}
\newcommand{\RI}[1]{\mathrm{RI}^{(#1)}}
\newcommand{\BA}[1]{\mathrm{BA}^{(#1)}}
\newcommand{\Res}[1]{\#\textsc{Asel}\left(#1\right)}
\newcommand{\Tog}[1]{\#\textsc{toggles}\left(#1\right)}
\newcommand{\bigO}[1]{\mathcal O(#1)}

\renewcommand{\(}{\left(}
\renewcommand{\)}{\right)}
\newcommand{\id}{\mathrm{Id}}

\title{The Power Light Cone of the Discrete Bak-Sneppen, Contact and other local processes}

\author{
    \begin{tabular}{c}
        Tom Bannink\footnote{Qusoft, CWI, University of Amsterdam, Science Park 123, 1098 XG Amsterdam, Netherlands.}\;\,\setcounter{footnote}{2}\footnote{Supported by the NWO Gravitation-grant QSC and NETWORKS-024.002.003, and by EU grant QuantAlgo.}\\
        \footnotesize{\texttt{bannink@cwi.nl}}
    \end{tabular}
    \and
    \begin{tabular}{c}
        Harry Buhrman\textsuperscript{$\ast$}\textsuperscript{\ddag}\\
        \footnotesize{\texttt{buhrman@cwi.nl}}
    \end{tabular}
    \and
    \begin{tabular}{c}
        András Gilyén\textsuperscript{$\ast$}\footnote{Supported by ERC Consolidator Grant QPROGRESS and QuantERA project QuantAlgo 680-91-034.}\\
        \footnotesize{\texttt{gilyen@cwi.nl}}
    \end{tabular}
    \and
    \begin{tabular}{c}
        Mario Szegedy\footnote{ {Alibaba Quantum Laboratory, Alibaba Group, Bellevue, WA 98004, USA, Also supported by NSF grants 1422102 and 1514164.}}\\\footnotesize{\texttt{mario.szegedy@alibaba-inc.com}}
    \end{tabular}
}
\date{\today}

\begin{document}
\maketitle{}
\begin{abstract}
    We consider a class of random processes on graphs that include the discrete Bak-Sneppen (DBS) process and the several versions of the contact process (CP), with a focus on the former. These processes are parametrized by a probability $0\leq p \leq 1$ that controls a local update rule. Numerical simulations reveal a phase transition when $p$ goes from 0 to 1. Analytically little is known about the phase transition threshold, even for one-dimensional chains.
	In this article we consider a power-series approach based on representing certain quantities, such as the survival probability or the expected number of steps per site to reach the steady state, as a power-series in $p$.
    We prove that the coefficients of those power series stabilize as the length $n$ of the chain grows. 
    This is a phenomenon that has been used in the physics community but was not yet proven. We show that for local events $A,B$ of which the support is a distance $d$ apart we have $\mathrm{cor}(A,B) = \mathcal{O}(p^d)$.
	The stabilization allows for the (exact) computation of coefficients for arbitrary large systems which can then be analyzed using the wide range of existing methods of power series analysis.
\end{abstract}

\section{Introduction}
In physics, critical behaviour involves systems in which correlations decay as a power law with distance. It is an important topic in many areas of physics and can also be found in stochastic processes on graphs. Often, such systems have a parameter (e.g. temperature) and when it is set to a critical value, the system exhibits critical behaviour. Power series expansion techniques have been used in the physics literature to numerically approximate critical values and associated exponents. It was often observed that the coefficients of such power series stabilize when the system size grows, and we provide a rigorous proof of this for a large class of stochastic processes.

Self-organized criticality is a name common to models where the critical behaviour is present but without the need of tuning a parameter. A simple model for evolution and self-organized criticality was proposed by Bak and Sneppen \cite{BakSneppen1993} in 1993. In this random process there are $n$ vertices on a cycle each representing a species. Every vertex has a fitness value in $[0,1]$ and the dynamics is defined as follows. Every time step, the vertex with the lowest fitness value is chosen and that vertex together with its two neighbors get replaced by three independent uniform random samples from $[0,1]$. The model exhibits self-organized criticality, as most of the fitness values automatically become distributed uniformly in $[f_c,1]$ for some critical value $0<f_c<1$. This process has received a lot of attention \cite{Boer1994,Marsili1994,Bak1996,Marsili1998}, and a discrete version of the process has been introduced in \cite{Barbay2001}. In the discrete Bak-Sneppen (DBS) process, the fitness values can only be $0$ or $1$. At every time step, choose a uniform random vertex with value $0$ and replace it and its two neighbors by three independent values, which are $0$ with probability $p$ and $1$ with probability $1-p$. The DBS process has a phase transition with associated critical value $p_c$ \cite{Meester2002,Bandt2005}.

The Bak-Sneppen process was originally described in the context of evolutionary biology but its study has much broader consequences, e.g., the process was rediscovered in the setting of theoretical computer science \cite{ResampleLimit}. To study the limits of a randomized algorithm for solving satisfiability, the discrete Bak-Sneppen process turned out to be a natural process to analyze.

The DBS process is closely related to the so-called contact process (CP), originally introduced in~\cite{Harris1974}. Sometimes referred to as the basic contact process, this process models the spreading of an epidemic on a graph where each vertex (an individual) can be healthy or infected. Infected individuals can become healthy (probability $1-p$), or infect a random neighbor (probability $p$). The contact process has also been studied in the context of interacting particle systems and many variants of it exist, such as a parity-preserving version~\cite{Inui1995} and a contact process that only infects in one direction~\cite{Tretyakov1997}. Depending on the particular flavor of the processes, the CP and DBS processes are closely related \cite{Bandt2005} and in certain cases have the same critical values. The processes are similar in the sense that vertices can be \emph{active} (fitness 0 or infected) or \emph{inactive} (fitness 1 or healthy). The dynamics only update the state in the neighborhood of active vertices with a simple local update rule. In this article we consider a wide class of processes that fit this description, and our proofs are valid in this general setting. We will, however, focus on the DBS process when we present explicit examples.

In this paper we take a power-series approach and represent several probabilities and expectation values as a power series in the parameter $p$. There is a wealth of physics literature on series analysis in the theory of critical phenomena, see for example \cite{Hunter1973,Baker1973,Hunter1979} for an overview. Processes typically only have a critical point when the system size is infinite, but numerical simulations often only allow for probing of finite systems. Our main theorem proves, for our general class of processes, that one can extract coefficients of the power series for an arbitrary large system by computing quantities in only a finite system. One can then apply series analysis techniques to these coefficients of the large system.
Series expansion techniques are not new, they have been extensively used for variants of the contact process as well as for closely related directed percolation models~\cite{Dickman1989,Jensen1993,Inui1995,Inui1996,Tretyakov1997,Inui1998,Katori1999} in order to extract information about critical values and exponents.
For example, in~\cite{Tretyakov1997} the contact process on a line is studied where infection only happens in one direction. In~\cite{Inui1995} a process is studied where the parity of the number of active vertices is preserved. In both articles, the power series of the survival probability is computed up to 12 terms and used to find estimates for the critical values and exponents. 
However, in all this work the stabilization of coefficients has been observed\footnote{Some work uses stabilization in the number of time steps instead of system size. However, for understanding the critical behavior, system size is the relevant parameter.} but not proven. 

\textbf{Our main contribution} is a definition of a general class of processes that encapsulates most of the above processes
(Definition \ref{def:plup}) and an in-depth understanding of the stabilization phenomenon, complete with a rigorous proof 
(Lemma~\ref{lemma:distancePower}, Theorem~\ref{thm:stabilization}). The results are illustrated with examples.

\textbf{Road map.} In Section~\ref{sec:stabilization} we will provide two example power series that exhibit the stabilization phenomenon. In Section~\ref{sec:plc} we will sketch our results without going into technicalities and explain the intuition behind them, something that we call the Power Light Cone.
In Section~\ref{sec:plup} we define our general class of processes in more detail and provide our theorems with their proofs.
In Section~\ref{sec:numerics} we apply our result to the DBS process, and we compute power-series coefficients for several quantities. As an application, we apply the method of Pad\'e approximants to extract an estimate for $p_c$ and we estimate a critical exponent that suggests that the DBS process is in the directed percolation universality class.

\subsection{Stabilization of coefficients} \label{sec:stabilization}

There are different ways of defining the DBS process. These definitions map to each other in straightforward ways, and only slightly differ in the notion of a ``time step''. For example, one can define a time step as picking a random vertex (not necessarily with value $0$) and then not do anything when it has value $1$ but still count it as a time step. To study the process for an infinite-sized system, one can consider a continuous-time version of this process with exponential clocks at every vertex. Resampling of a vertex and its neighbors happens when the clock of the vertex rings \emph{and} the current value of the vertex is zero. 

In this section we consider the discrete-time process on a finite system, where a $0$-vertex is picked in every time step, and the process \emph{terminates} when an all-$1$ state is reached. Furthermore, we will always refer to the vertex-state $0$ as \emph{active} and $1$ as \emph{inactive}.

From numerical simulations it is apparent that there is a phase transition in the DBS process when $p$ goes from 0 to 1. There is some critical probability $p_c$ such that for $p < p_c$ the active vertices quickly die out and the system is pushed toward a state with no active vertices. However for $p > p_c$, the active vertices have the overhand and dominate the system. This phase transition can clearly be seen in Figure~\ref{fig:plotreachend}
from two different perspectives:
(a) We have computed the expected number of steps per vertex, on a cycle of length $n$, to reach the all-inactive state, where the vertices are initialized i.i.d. to active with probability $p$. (b) We have computed the probability that the end of a (non-periodic) chain is activated when the process is started with only one active vertex on the other end.

\begin{figure}
	\begin{center}
		\includegraphics[draft=false,width=0.48\textwidth]{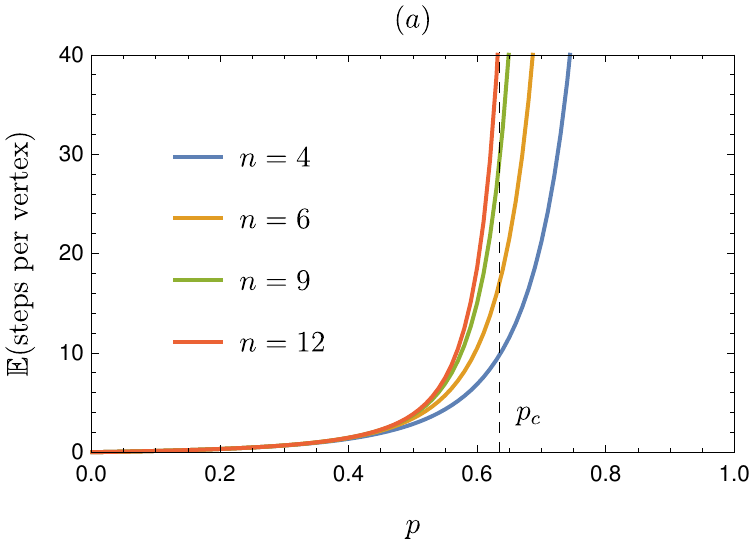}
        \;
		\includegraphics[draft=false,width=0.48\textwidth]{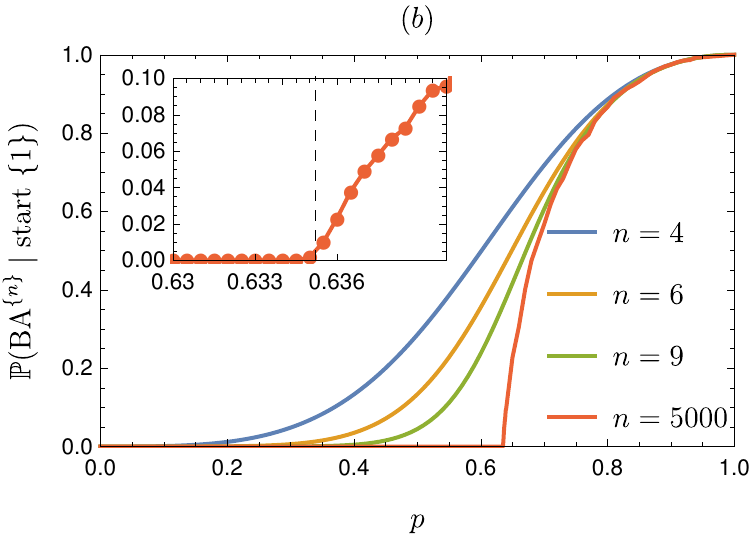}
        \caption{\label{fig:plotreachend} (a) Plot of $R_{(n)}(p)$, see~\eqref{eq:informalRseries}, the expected number of steps per vertex before the all-inactive state is reached, for the DBS process on a cycle with $n$ vertices. The process was started with independent values for each vertex, being active with probability $p$. (b) Plot of $S_{[n]}(p)$, see~\eqref{eq:informalQseries}, the probability to `reach' the other side of the system: the DBS process on a non-periodic chain of size $n$ is started with a single active vertex at position $1$ (denoted by $\text{start }\{1\}$) and we plot the probability that vertex $n$ ever \underline{b}ecomes \underline{a}ctive (denoted $\BA{n}$) before the all-inactive state is reached. For $n=5000$ the result was obtained with a Monte Carlo simulation. For the lower $n$, the results were computed symbolically. The inset shows a zoomed in version of the Monte Carlo data, showing that $p_c \approx 0.635$.
        }
	\end{center}
\end{figure}
Let us write these functions as a power-series in $p$ and in $q=1-p$ respectively.
\begin{align}
    R_{(n)}(p) &:= \frac{1}{n}\mathbb{E}(\text{total steps} \mid \text{start i.i.d.}) = \sum_{k=0}^{\infty} a^{(n)}_k p^k, \label{eq:informalRseries} \\
    S_{[n]}(q) &:= \P( \text{vertex $n$ becomes active} \mid \text{start } \{1\}) = \sum_{k=0}^{\infty} b^{[n]}_k q^k .\label{eq:informalQseries}
\end{align}
We will define these functions in more detail in Section~\ref{sec:numerics}, where we show, amongst other things, that they are rational functions for each $n$. For example
\begin{align*}
    R_{(4)}(p) &= \frac{p(6-12p+10p^2-3p^3)}{6(1-p)^4}
                = \frac{(1-q)(1+q+q^2+3q^3)}{6q^4} .
\end{align*}
\begin{figure}
    \begin{center}
        \includegraphics[draft=false,width=0.70\textwidth]{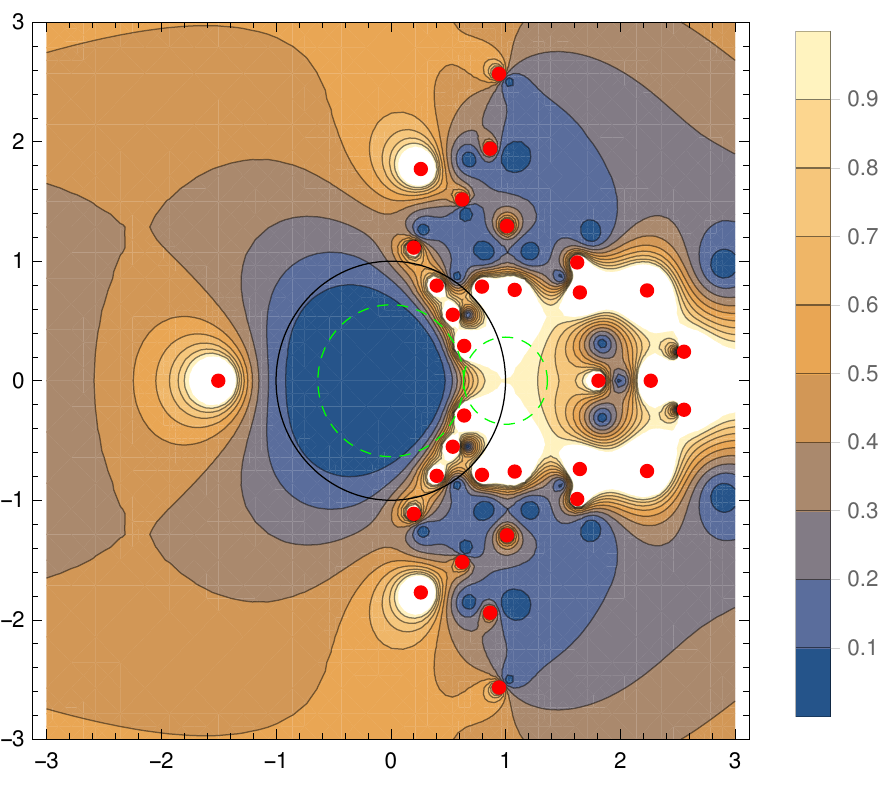}
        \caption{\label{fig:rationalplot} Plot of the function $\vert S_{(6)}(p) \vert$, defined in Equation~\eqref{eq:informalQseries}, over the complex plane with $p=0$ at the origin. The poles of the function are shown as red dots. The unit circle is shown in black, and the dashed green circles have radius $p_c$ around the origin, and radius $1-p_c$ around $p=1$.}
    \end{center}
\end{figure}
Although they only have an operational meaning for $p\in[0,1]$, we give a plot of such a function over the complex plane, see Figure~\ref{fig:rationalplot}. The plot shows the poles of $S_{(6)}(p)$, which seem to approach the value $p_c$ on the real line (for larger $n$ see Figure~\ref{fig:Qplotpoles}).
Something similar happens for partition functions in statistical physics. The partition function is usually in the denominator of observable physical quantities, so that its zeroes are the poles of such quantities. A classic result on the partition function for certain gasses~\cite{Yang1952} shows that when an open region around the real axis is free of zeroes, then many physical quantities are analytic in that region and therefore there is no phase transition. In~\cite{Regts2017} the hardcore model on graphs with bounded degree is studied, and it is proven that the partition function has zeroes in the complex plane arbitrary close to the critical point.

For now, we simply want to highlight the behaviour of the coefficients $a^{(n)}_k$ and $b^{[n]}_k$. Table~\ref{tab:coeffs} and Table~\ref{tab:Qcoeffs} show numerical values of the coefficients $a^{(n)}_k$ and $b^{[n]}_k$ respectively.\footnote{At first sight one is tempted to conjecture that the coefficients $a^{(n)}_k$ are all non-negative and are monotone increasing with $n$. Unfortunately neither of these conjectures hold since $a^{(10)}_{1114}<0$. We found this counterexample by observing that the radius of convergence for $R_{10}(p)$ is less than $0.96$, which considering that $R_{10}(p)$ is bounded on $[0,0.96]$ implies that there must be a negative coefficient in its power series.}
\begin{table}
	\centering
	\caption{Table of the coefficients $a^{(n)}_k$ of the power series defined in Equation~\eqref{eq:informalRseries}. Although displayed with finite precision, they were computed symbolically. }
	\label{tab:coeffs}
	\resizebox{\columnwidth}{!}{%
		\begin{tabular}{c|cccccccccccccccccc}
			\backslashbox[10mm]{$n$}{$k$} & 0 & 1 & 2 & 3 & 4 & 5 & 6 & 7 & 8 & 9 & 10 & 11 & 12 & 13 & 14 & 15 & 16 & 17 \\		\hline
			3 &	0 & 1 & \cellcolor{blue!25}2 & 3+1/3 & 5.00 & 7.00 & 9.33 & 12.00 & 15.00 & 18.33 & 22.00 & 26.00 & 30.33 & 35.00 & 40.00 & 45.333 & 51.000 & 57.000 \\
			4 &	0 & 1 & 2 & \cellcolor{blue!25}3+2/3 & 6.16 & 9.66 & 14.3 & 20.33 & 27.83 & 37.00 & 48.00 & 61.00 & 76.16 & 93.66 & 113.6 & 136.33 & 161.83 & 190.33 \\
			5 &	0 & 1 & 2 & 3+2/3 & \cellcolor{blue!25}6.44 & 10.8 & 17.3 & 26.65 & 39.43 & 56.48 & 78.65 & 106.9 & 142.2 & 185.8 & 238.7 & 302.41 & 378.05 & 467.13 \\
			6 &	0 & 1 & 2 & 3+2/3 & 6.44 & \cellcolor{blue!25}11.0 & 18.5 & 30.02 & 47.10 & 71.68 & 106.0 & 152.9 & 215.4 & 297.4 & 403.1 & 537.21 & 705.25 & 913.31 \\
			7 &	0 & 1 & 2 & 3+2/3 & 6.44 & 11.0 & \cellcolor{blue!25}18.7 & 31.21 & 50.83 & 80.80 & 125.3 & 189.7 & 280.8 & 407.0 & 578.6 & 808.13 & 1110.2 & 1502.6 \\
			8 &	0 & 1 & 2 & 3+2/3 & 6.44 & 11.0 & 18.7 & \cellcolor{blue!25}31.44 & 52.08 & 84.95 & 136.0 & 213.6 & 328.9 & 496.5 & 735.6 & 1070.7 & 1532.5 & 2159.5 \\
			9 &	0 & 1 & 2 & 3+2/3 & 6.44 & 11.0 & 18.7 & 31.44 & \cellcolor{blue!25}52.30 & 86.27 & 140.7 & 226.3 & 358.4 & 558.4 & 855.4 & 1289.0 & 1911.5 & 2791.4 \\
			10&	0 & 1 & 2 & 3+2/3 & 6.44 & 11.0 & 18.7 & 31.44 & 52.30 & \cellcolor{blue!25}86.49 & 142.1 & 231.6 & 373.4 & 594.8 & 934.4 & 1447.1 & 2209.0 & 3324.6 \\
			\vdots \\
            18& 0 & 1 & 2 & 3+2/3 & 6.44 & 11.08 & 18.76 & 31.45 & 52.31 & 86.49 & 142.33 & 233.31 & 381.17 & 621.02 & 1009.38 & 1637.13 & 2650.56 & \cellcolor{blue!25}4284.31 
		\end{tabular}
	}
\end{table}
\begin{table}
	\centering
	\caption{Table of the coefficients $b^{[n]}_k$ of the power series defined in Equation~\eqref{eq:informalQseries}. Although displayed with finite precision, they were computed symbolically.}
	\label{tab:Qcoeffs}
	\resizebox{\columnwidth}{!}{%
		\begin{tabular}{c|cccccccccccccccc}
			\backslashbox[10mm]{$n$}{$k$} & 0 & 1 & 2 & 3 & 4 & 5 & 6 & 7 & 8 & 9 & 10 & 11 & 12 & 13 & 14 & 15 \\ \hline
             3    & 1 & 0 & \cellcolor{blue!25}-2 & -2 & 0 & 4 & 6 & 2 & -8 & -16 & -10 & 14 & 40 & 36 & -18 & -94\\
             4    & 1 & 0 & -2 & \cellcolor{blue!25}-4 & -3.50 & 5.75 & 22.25 & 31.31 & 1.91 & -89.13 & -200.21 & -171.80 & 220.35 & 992.97 & 1513.93 & 352.89\\
             5    & 1 & 0 & -2 & -4 & \cellcolor{blue!25}-8.25 & -2.46 & 23.94 & 76.89 & 127.93 & 50.05 & -357.65 & -1208.56 & -2034.75 & -1004.08 & 5178.29 & 18688.02\\
             6    & 1 & 0 & -2 & -4 & -8.25 & \cellcolor{blue!25}-13.65 & 2.11 & 76.61 & 239.51 & 422.23 & 325.17 & -860.92 & -4423.36 & -10847.19 & -15746.83 & -2393.70\\
             7    & 1 & 0 & -2 & -4 & -8.25 & -13.65 & \cellcolor{blue!25}-24.58 & 19.20 & 221.39 & 689.37 & 1325.77 & 1325.82 & -1515.32 & -12291.20 & -38583.12 & -81814.55\\
             8    & 1 & 0 & -2 & -4 & -8.25 & -13.65 & -24.58 & \cellcolor{blue!25}-44.71 & 69.28 & 599.01 & 1939.94 & 3969.88 & 4778.48 & -1873.44 & -30668.24 & -112066.49\\
             9    & 1 & 0 & -2 & -4 & -8.25 & -13.65 & -24.58 & -44.71 & \cellcolor{blue!25}-84.21 & 196.95 & 1590.70 & 5328.67 & 11662.68 & 15977.89 & 1408.10 & -75058.95\\
             10   & 1 & 0 & -2 & -4 & -8.25 & -13.65 & -24.58 & -44.71 & -84.21 & \cellcolor{blue!25}-172.29 & 531.63 & 4131.50 & 14490.20 & 33615.60 & 51447.00 & 22246.40
		\end{tabular}
	}
\end{table}
A quick look at the table immediately reveals the stabilization of coefficients:
\begin{align*}
    a^{(n)}_k &= a^{(k+1)}_k \qquad \forall n \geq k+1 \qquad \text{ and }& b^{[n]}_k &= b^{(k+1)}_k \qquad \forall n \geq k+1 .
\end{align*}
We now know the first few terms of the power series for arbitrary large systems and we can proceed to use methods of series analysis. By applying the method of Pad\'e approximants, we can estimate $p_c \approx 0.6352$. More details on this can be found in Section~\ref{sec:numerics}.

\subsection{Locality of update rule implies stabilization} 
\label{sec:plc}

Our proofs are based on an observation that we call the Power Light Cone. Let $X$ be a set of vertices, and let $L_X$ be an event that is local on $X$, meaning that the event depends only on what happens to the vertices in $X$. For example, when $X=\{v_0\}$ and $L_X$ is the event that vertex $v_0$ is picked at least $r$ times, then $L_X$ is local on $X$. In Section~\ref{sec:plup} we will give a more precise definition of local events.
We now wish to compare the probability $\P(L_X)$ when the process is initialized in two different starting states, $A$ and $A'$. When $A$ and $A'$ differ only on vertices that are at least a distance $d$ away from $X$, then we have
\begin{align*}
    \P(L_X \mid \text{start in } A) - \P(L_X \mid \text{start in }A') = \bigO{ p^d } .
\end{align*}
By the notation $\bigO{p^d}$ we mean that when this quantity is written as power series in $p$, then the first $d-1$ terms of the series are zero. It only has non-zero terms of order $p^d$ and higher, i.e. the two probabilities agree on at least the first $d-1$ terms of their power series.
This is the essence of the Power Light Cone. A vertex that is a distance $d$ away from the set $X$ will only influence probabilities and expectation values of $X$-local events with terms of order $p^d$ or higher. The intuition behind this is that the probability of a single activation is $\bigO{p}$ and in order for such a vertex to influence the state of a vertex in $X$, it needs to form a chain of activations of size $d$ to reach $X$.
This observation will also allow us to compare the process on systems of different sizes. 
\begin{lemma}[Informal version of Lemma~\ref{lemma:distancePower}]
    Let $G$ and $G'$ be two graphs and let $X$ be a set of vertices present in both graphs such that the $d$-neighborhood of $X$ and the local update process (a single update may only affect a vertex and its neighbors) on it is the same in both graphs. Then for any event $L_X$ that is local on $X$ we have
    \begin{align*}
        \P_{G}(L_X) = \P_{G'}(L_X) + \bigO{p^d} .
    \end{align*}
\end{lemma}

This idea applies to expectation values as well. Consider the expected number of steps per vertex on a cycle. By translation invariance, we have
\begin{align*}
    \frac{1}{n} \mathbb{E}(\text{total steps}) = \mathbb{E}(\text{\#times vertex 1 was picked})
\end{align*}
making it a $\{1\}$-local quantity. If we add an extra vertex to the cycle, the expectation value only changes by a term of order $\bigO{p^{n/2}}$ since the new vertex has distance $n/2$ to vertex $1$. 

\section{Parametrized local-update processes} \label{sec:plup}

The class of \emph{parametrized (discrete) local-update processes}, introduced in this section, includes the DBS, the CP and many other natural processes. We prove a general `stabilization of the coefficients theorem' for them, suggesting the usefulness of the power-series approach for members of the class.

Let $G=(V,E)$ be an undirected graph with vertex set $V$ and edge set $E$. We consider processes where every vertex of $G$ is either \emph{active} or \emph{inactive}. 
A \emph{state} is a configuration of active/inactive vertices, denoted by the subset of \emph{active} vertices $A\subseteq V$. For $v\in V$ let us denote by $\Gamma(v)$ the neighbors of $v$ in $G$ including $v$ itself. A local update process in each discrete time step picks a random active vertex $v\in A$ and resamples the state of its neighbors $\Gamma(v)$. If the state is $\emptyset$ (there are no active vertices) then the process stops and all vertices remain inactive afterwards.

\begin{definition}[PLUP - Parametrized local-update process] \label{def:plup}
    We say that $M_G$ is a \emph{parametrized local-update process} on the graph $G=(V,E)$ with parameter $p\in[0,1]$ if it is a time-independent Markov chain on the state space $\{\text{inactive},\text{active}\}^{V}$ that satisfies the following:
	\begin{enumerate}[label=(\roman*)]
        \item \label{prop:indepInit}
            \textbf{Initial state.} The initial value of a vertex is picked independently from the other vertices.
            The probability of initializing $v\in V$ as active is a polynomial in $p$ with constant term equal to zero.\footnote{The zero constant term is used, for example, in Lemma~\ref{lemma:activationpower}. The independence is used in Lemma~\ref{lemma:splitting}.}
        \item \label{prop:weightedPick}
            \textbf{Selection dynamics.} Each vertex $v\in V$ has a fixed positive weight $w_v$. A vertex $v\in V$ is selected using one of the three rules below, and if the selected vertex was active, then its neighborhood $\Gamma(v)$ is resampled using the parametrized local-update rule of vertex~$v$.\footnote{The properties of the selection dynamics are used in the proof of Lemma~\ref{lemma:splitting}}
            \begin{enumerate}
                \item \label{it:active}\textbf{Discrete-time active sampling.} In each discrete time step, an active vertex $v\in A$ is selected with probability $\frac{w_v}{\sum_{u\in A}w_u}$, where $A$ is the current state. 
                \item \label{it:all}\textbf{Discrete-time random sampling.} In each discrete time step, a vertex $v\in V$ is selected with probability $\frac{w_v}{\sum_{u \in V} w_u}$.
                \item \label{it:clock}\textbf{Continuous-time clocks.} Every vertex $v\in V$ has an exponential clock with rate $w_v$. When a clock rings, that vertex is selected, and a new clock is set up for the vertex.
            \end{enumerate}
    \item \label{prop:locUpdateRule} 
    \textbf{Update dynamics.} The parametrized local-update rule of a vertex $v\in V$ describes a (time-independent) probabilistic transition from state $A$ to $A'$ such that the states only differ on the neighborhood $\Gamma(v)$, i.e., $A\bigtriangleup A'\subseteq \Gamma(v)$. The probability $P_R$ of obtaining active vertices $R=A'\cap \Gamma(v)$ is independent of $A\setminus \Gamma(v)$. The probability $P_R$ is a polynomial in $p$ such that for $p = 0$ we get $A' \subsetneq A$ with probability $1$, i.e., when any previously inactive vertex becomes active ($\;|A' \setminus A| > 0$) \textit{or} when $A'=A$ then the constant term in $P_R$ must be zero.\footnote{The condition $|A'\setminus A|>0 \implies P_R = \bigO{p}$ is used in the proof of Lemma~\ref{lemma:activationpower}: a fresh activation is at least one power of $p$ so you need $p^k$ to cover a distance $k$. The extra condition $A'=A \implies P_R = \bigO{p}$ is used for absolute convergence in Claim~\ref{claim:finitecontribution} because without it you can have infinitely many paths with a finite power of $p$.}
    \item \textbf{Termination.} The process terminates when the all-inactive state $\emptyset$ is reached.
	\end{enumerate}
\end{definition}

We write $\P_G$ and $\E_G$ for probabilities and expectation values associated to the PLUP $M_G$.

\begin{definition}[Local events] \label{def:events}
	Let $G=(V,E)$ be a graph and let $M_G$ be a PLUP. Let $S\subseteq V$ be any subset of vertices, and let $v\in V$ be any vertex.
	\begin{itemize}
        \item Let $\II{S}$ be the event that all vertices in $S$ get \underline{i}nitialized as \underline{i}nactive.
        \item Let $\RI{S}$ be the event that all vertices in $S$ \underline{r}emain \underline{i}nactive during the entire process.
        \item Define $\BA{S}$ as the complement of $\RI{S}$: the event that there exists a vertex in $S$ that \underline{b}ecomes \underline{a}ctive at some point during the process, including initialization.
        \item Let $\Res{v}$ be the number of times that $v$ was \underline{sel}ected while it was \underline{a}ctive.
        \item Let $\Tog{v}$ be the number of times that the value of $v$ was changed. 
        \item We say an event $L$ is \textbf{local} on the vertex set $S$ if it is in the sigma algebra\footnote{If $p<1$, then the process terminates with probability $1$. If we remove the $0$ probability event of not terminating, then the discrete nature of the process implies that the probability space has a discrete set of atoms see, e.g., Definition~\ref{def:paths1} and Equation~\eqref{eq:sumtpaths}, therefore one does not need to worry about difficulties arising from $\sigma$-additivity.} generated by the events
		\begin{align*}
            \RI{v} \; , \; \BA{v} , \; (\Res{v} = k) \; , \; (\Tog{v} = k) \quad\colon v\in S, 0\leq k < \infty.
		\end{align*}
	\end{itemize}
\end{definition}

\begin{lemma}[Time equivalence] \label{lemma:time}
    The three versions of the selection dynamics of a PLUP, desribed in property~\ref{prop:weightedPick} of Definition~\ref{def:plup}, are equivalent for local events. That is, for any local event $L$ the probability $\P(L)$ is independent of the chosen selection dynamics in property~\ref{prop:weightedPick}.
\end{lemma}
\begin{proof}
    The three selection dynamics only differ in the counting of time, and the presence of self loops in the Markov Chain. The definition of local events only includes events that are independent of the way time is counted. They only depend on which active vertices are selected and the changes to the state of the graph.
    
    It is easy to see that \ref{it:all} implements the dynamics of \ref{it:active} via rejection sampling, therefore they give rise to the same probabilities. One can also see that on a finite graph the selection rule \ref{it:clock} induces the same selection rule as \ref{it:all}. This is because the exponential clocks induce a Poisson process at each vertex. The $n$ independent Poisson processes with rates $w_v$ are equivalent to one single Poisson process with rate $W=\sum_{v\in V}w_v$ but where each point of the single process is of type $v$ with probability $w_v / W$. One can simulate \ref{it:clock} by sampling a time value from an exponential distribution with parameter $W$ and then sampling a random vertex with probability $w_v/W$ (as in \ref{it:all}). Since the time is not relevant for local events we can ignore the sampled time value and this gives rise to the same probabilities.
\end{proof}

Our lemmas and theorems only concern local events and therefore we can use any one of the three selection dynamics when proving them.

\begin{definition}[Induced process]
    Suppose that $V'\subseteq V$, then we define the induced process $M_{G'}$ on the induced subgraph $G'=(V',E')$ 
    such that we run the process on $G$ and after each step we deactivate all vertices in $V\setminus V'$. We can then view this as a process on $G'$.
    Let $L$ be a local event on $V'$. We denote the probability of $L$ under the induced process $M_{G'}$ with $\P_{G'}(L)$. Similarly we use the notation $\E_{G'}$ for expectation values induced by the process $M_{G'}$.
\end{definition}
It is easy to see that the induced process of a PLUP is also a PLUP.

\begin{definition}[Graph definitions]
	Let $G=(V,E)$ be a graph, $S\subseteq V$ be any subset of vertices and $v\in V$ be any vertex.
	\begin{itemize}	
		\item Define $G\setminus S$ as the induced subgraph on $V\setminus S$ and $G\cap S$ as the induced subgraph on $S$.
		\item Define the $d$-neighbourhood $\Gamma(S,d)$ of $S$ as the set of vertices that are connected to $S$ with a path of length at most $d$. In particular $\Gamma(\{v\};1)=\Gamma(v)$.
		\item Define the distant-$k$ boundary $\overline{\partial}(S,k):=\Gamma(S,k)\setminus \Gamma(S,k-1)$ as the set of vertices lying at exactly distance $k$ from $S$, and let 	 $\overline{\partial}S:=\overline{\partial}(S,1)$.
	\end{itemize}
\end{definition}

The following lemma says that if a set $S$ splits the graph into two parts, then those two parts become independent under the condition that the vertices in $S$ never become active.
\begin{center}
	\includegraphics[scale=0.8,draft=false]{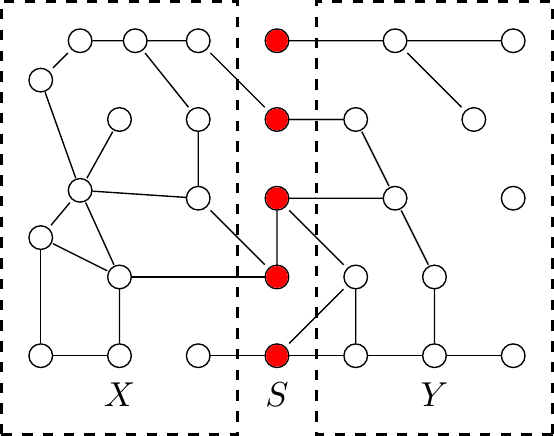}
\end{center}
\begin{lemma}[Splitting lemma]\label{lemma:splitting}
	Let $M_G$ be a parametrized local-update process on the graph $G=(V,E)$. Let $S,X,Y\subseteq V$ be a partition of the vertices, such that $X$ and $Y$ are disconnected in the graph $G\setminus S$. Furthermore, let $L_X$ and $L_Y$ be local events on $X$ and $Y$ respectively. Then we have
	\begin{align*}
        \P_{G}(\RI{S} \cap L_X \cap L_Y \mid \II{S})
        &=
        \P_{G\setminus Y}(\RI{S} \cap L_X \mid \II{S})
        \; \cdot \;
        \P_{G\setminus X}(\RI{S} \cap L_Y \mid \II{S}) .
	\end{align*}
\end{lemma}
The condition of initializing $S$ to inactive is present only to prevent counting the initialization probabilities twice. Equivalently we could write the condition only once:
\begin{align*}
	\P_{G}(\RI{S} \cap L_X \cap L_Y)
	&=
	\P_{G\setminus Y}(\RI{S} \cap L_X)
	\; \cdot \;
    \P_{G\setminus X}(\RI{S} \cap L_Y \mid \II{S}) ,
\end{align*}
and by Bayes rule we also have
\begin{align*}
    \P_{G}(L_X \cap L_Y \mid \RI{S})
	&=
    \P_{G\setminus Y}(L_X \mid \RI{S})
	\; \cdot \;
	\P_{G\setminus X}(L_Y \mid \RI{S}) .
\end{align*}

\begin{proof}
    We will use the `continuous-time clocks' version of selection dynamics (PLUP property~\ref{it:clock}). By Lemma~\ref{lemma:time} the statement will then hold for all versions.
    We proceed with a coupling argument. There are three processes, one on $G$ and the induced ones on $G\setminus Y$ and $G\setminus X$. We couple them by letting all three processes use the same source of randomness. Every vertex in $G$ has an exponential clock that is shared by all three processes, and the randomness used for the local updates for each vertex will also come from the same source. This means that when the clock of a vertex $v$ rings, and the neighborhood $\Gamma(v)$ is equal in different processes, then the update result will also be equal.
    Now we simply observe that $L_X \cap L_Y \cap \RI{S}$ holds in the $G$-process if and only if $L_X \cap \RI{S}$ holds in the $(G\setminus Y)$-process \emph{and} $L_Y \cap \RI{S}$ holds in the $(G\setminus X)$-process.
    This is because all vertices in $S$ are initialized as inactive (all three probabilities are conditioned on this), so a vertex in $S$ can only be activated by an update from a vertex in $X$ or $Y$. To check if the event $\RI{S}$ holds, it is sufficient to trace the process up to the first activation of a vertex in $S$. Before this first activation, anything that happends to the vertices in $X$ only depends on the clocks and updates of vertices in $X$, and similar for $Y$. Since $S$ splits $X$ and $Y$ in disconnected parts, these parts can not influence each other unless a vertex in $S$ is activated. Because of the coupling, the evolution of the $X$ vertices in $G\setminus Y$ will be exactly the same as the evolution in $G$, and similar for $Y$. Once a vertex in $S$ \emph{does} get activated, the evolution of the three processes is no longer the same but in that case the event $\RI{S}$ does not hold, regardless of any further updates in any system. The clocks and updates of each vertex are independent sources of randomness, and when $\RI{S}$ holds then all the randomness of the $S$ vertices is ignored. Therefore the probability of $\RI{S}$ in the $(G\setminus Y)$-process and $(G\setminus X)$-process depends only on independent random variables and we get the required equality.
\end{proof}

\subsection{Power Light Cone results} \label{sec:plcresults}

Now we present the results that exhibit the power light cone. The intuition is that if two vertices have distance $d$ in the graph, then the only way they can affect each other is that an interaction chain is forming between them, meaning that every vertex gets activated at least once in between them.

When we write $f(p) = \bigO{p^k}$ for some function $f$ then we mean the following: when $f(p)$ is written as a power-series in $p$, i.e., $f(p) = \sum_{i=0}^{\infty} \alpha_i p^i$, then $\alpha_i=0$ for $0\leq i \leq k-1$.

\begin{restatable}{lemma}{activationpower}\label{lemma:activationpower}
    Let $M_G$ be a parametrized local-update process on the graph $G=(V,E)$. Let $X\subseteq V$ be a subset of vertices and let $E$ be an event. If $E \subseteq \bigcap_{v \in X} \BA{v}$, then $\P(E) = \bigO{p^{|X|}}$.
\end{restatable}
\begin{proof}
    When $E$ holds, all vertices in $X$ become active. By PLUP property~\ref{prop:indepInit} any activation in the initial state is $\bigO{p}$ and by property~\ref{prop:locUpdateRule} any subsequent activation is also $\bigO{p}$. Therefore, for any path $\xi$ of the Markov Chain with $\xi \in E$ we have $\P(\xi) = \bigO{p^{|X|}}$, where $\P(\xi)$ is a polynomial in $p$. We have $\P(E) = \sum_{\xi \in E} \P(\xi)$ by definition. This is a sum over infinitely many polynomials, and by considering $\P(E)$ as a power series in $p$ we are effectively regrouping terms in this sum. In Appendix~\ref{apx:absconvergence} we prove that these regroupings are allowed by proving the absolute convergence of certain series.
\end{proof}

\begin{lemma}[Graph surgery]\label{lemma:distancePower}
    Let $M_G$ be a parametrized local-update process on the graph $G=(V,E)$. If $X,Y\subseteq V$, $X\cap Y=\emptyset$ and $L_X$ is a local event on $X$, then
	$$\P_{G}(L_X)-\P_{G\setminus Y}(L_X)=\bigO{p^{d(X,Y)}}.$$
	\anote{Should be true with $+1$ in the degree, when $d(X,Y)>0$!}
\end{lemma}
\begin{proof}
	We can assume without loss of generality, that $X\neq \emptyset\neq Y$, otherwise the statement is trivial. Also we can assume without loss of generality that $d(X,Y)\leq \infty$, i.e., $X,Y$ are in the same connected component of $G$, otherwise we can use Lemma~\ref{lemma:splitting} with $S=\emptyset$.
	
    The proof goes by induction on $d(X,Y)$. For the base case, $d(X,Y)=1$, first note that when $p=0$, the process initializes everything to inactive by property~\ref{prop:indepInit}. Depending on whether this atomic event is included in $L_X$, the probability $\P(L_X)$ for $p=0$ (i.e. the constant term) is either 0 or 1 and independent of the graph.

	Now we show the inductive step, assuming we know the statement for $d$, and that $d(X,Y)=d+1$.
	First we assume, that $\RI{X}\subseteq\overline{L_X}$, i.e., $L_X\subseteq \BA{X}$.
    Define
    \begin{align*}
        L_X^i     &:= L_X \cap \RI{\overline{\partial}(X,i)} \cap \bigcap_{j\in[i-1]} \BA{\overline{\partial}(X,j)} \qquad\text{for}\quad i \in [d],\\
        L_X^{d+1} &:= L_X \cap \bigcap_{j\in [d]} \BA{\overline{\partial}(X,j)} .
    \end{align*}
    When $L_X^i$ holds, all vertices at distance $i$ remain inactive, but for all $j \leq i-1$ there exists a vertex at distance $j$ that become active.
    These events form a partition $L_X=\dot\bigcup_{i\in [d+1]}L_X^{i}$. Below we depict $L_X^{i}$ graphically:
    \begin{center}
        \includegraphics[draft=false]{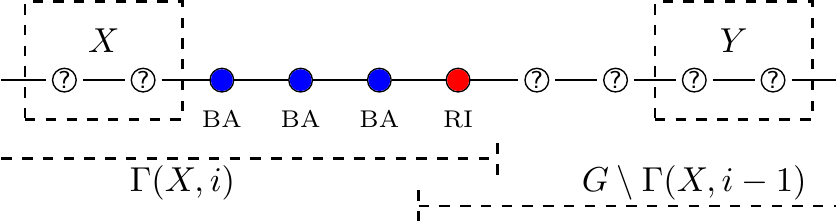}
    \end{center}
    It is easy to see that for all $i\in[d+1]$ we have $L_X^{i}\subseteq\BA{X}\cap\bigcap_{j\in[i-1]}\BA{\overline{\partial}(X,j)}$, and therefore by Lemma~\ref{lemma:activationpower} we get
	\begin{equation}\label{eq:AXorder}
        \P_G(L_X^{i} \mid \II{\overline{\partial}(X,i)} )=\bigO{p^{i}}.
	\end{equation}
    Now we use, for all $i \in [d]$, the Splitting lemma~\ref{lemma:splitting} with $S=\overline{\partial}(X,i)$ to split $\Gamma(X,i-1)$ from $G\setminus \Gamma(X,i)$. We get
	\begin{align}
	\P_G(L_X^{i})
	&=\P_{\Gamma(X,i)}(L_X^{i} \mid \II{\overline{\partial}(X,i)})\cdot \P_{G\setminus \Gamma(X,i-1)}(\RI{\overline{\partial}(X,i)}) \tag{by Lemma~\ref{lemma:splitting}}\\
	&=\P_{\Gamma(X,i)}(L_X^{i} \mid \II{\overline{\partial}(X,i)})\cdot \left(\P_{G\setminus Y\setminus \Gamma(X,i-1)}(\RI{\overline{\partial}(X,i)})+\bigO{p^{d+1-i}}\right) \tag{by induction}\\
	&=\P_{\Gamma(X,i)}(L_X^{i} \mid \II{\overline{\partial}(X,i)})\cdot \P_{G\setminus Y\setminus \Gamma(X,i-1)}(\RI{\overline{\partial}(X,i)})+\bigO{p^{d+1}} \tag{by equation \eqref{eq:AXorder}}\\
	&=\P_{G\setminus Y}(L_X^{i})+\bigO{p^{d+1}} \tag{by Lemma~\ref{lemma:splitting}}\\
	&=\P_{G\setminus Y}(L_X^{i})+\bigO{p^{d(X,Y)}}. \label{eq:indStep}
	\end{align}
	Therefore 
    \begin{align*}
	    \P_G(L_X)
        &\overset{\eqref{eq:AXorder}}{=}\sum_{i\in[d]}\P_G(L_X^{i})+\bigO{p^{d(X,Y)}}
	    \overset{\eqref{eq:indStep}}{=}\sum_{i\in[d]}\P_{G\setminus Y}(L_X^{i})+\bigO{p^{d(X,Y)}} \\
        &\overset{\eqref{eq:AXorder}}{=}\P_{G\setminus Y}(L_X)+\bigO{p^{d(X,Y)}}.
    \end{align*}
    We finish the proof by observing that by $\RI{X}$ is an atomic event of the sigma algebra of the local events of $X$, so if $\RI{X}\nsubseteq\overline{L_X}$,
	then we necessarily have $\RI{X}\subseteq L_X$. Therefore we can use the above proof with $C_X:=\overline{L_X}$ and use that $\P(L_X)=1-\P(C_X)$.
\end{proof}

\begin{restatable}[Decay of correlations]{corollary}{decay} \label{cor:decay}
	Let $M_G$ be a parametrized local-update process on the graph $G=(V,E)$. If $X,Y\subseteq V$ and $L_X, L_Y$ are local events on $X$ and $Y$ respectively, then
	\begin{equation}\label{eq:LCovariance}
		\mathrm{Cov}(L_X,L_Y)=\P_{G}(L_X\cap L_Y)-\P_{G}(L_X)\P_{G}(L_Y)=\bigO{p^{d(X,Y)-1}},
	\end{equation}
	and
	\begin{equation}\label{eq:BACovariance}
		\P_{G}(\BA{X}\cap \BA{Y})-\P_{G}(\BA{X})\P_{G}(\BA{Y})=\bigO{p^{d(X,Y)+1}}.
	\end{equation}
\end{restatable}
The proof of this lemma is analogous to the proof of Lemma~\ref{lemma:distancePower} and can be found in Appendix~\ref{apx:correlations}.

In order to state our general result about the stabilization of the coefficients in the power series we define a notion of isomorphism between different PLUPs. 
\begin{definition}[PLUP isomorphism]
	We say that the PLUPs $M_G$ and $M_{G'}$ are isomorphic with the fixed sets $X,X'$ if there is a graph isomorphism $i: G\rightarrow G'$ such that $i(X)=X'$, moreover the probability of transitioning in one step from a state $A$ to $A'$ is preserved under the isomorphism:
	$$
		\P_G(A \text{ is transformed to }A')=\P_{G'}(i(A) \text{ is transformed to }i(A')),
	$$ 
	and similarly the probability of initialising to a particular state $A$ is preserved:
	$$
	\P_G(\text{graph state is initially }A)=\P_{G'}(\text{graph state is initially }i(A')).
	$$ 	
	We denote such an isomorphism relation by
	$$
		M_{G}
		\underset{X'}{\overset{X}{\simeq}}
		M_{G'}.
	$$
\end{definition}

Now we define convergent families of PLUPs. Our requirements for such a family of processes imply that the underlying graphs converge to a common graph limit, also called graphing, therefore justifying the term ``convergent''.
Examples of convergent families of PLUPs include DBS and CP on toruses of any dimension, when the limit graphing is just the infinite grid. Less regular examples are also included, such as toroid ladder graphs or discrete Möbius strips of fixed width. 

\begin{definition}[Convergent family of PLUPs]
    We say that family $\{ (M_{G_j},v_j) \colon j\in \mathbb{N}\}$ of rooted PLUPs is convergent, if for all $d\in \mathbb{N}$ and for all $j,k \geq d$ we have $M_{\Gamma_{G_j}\left(\{v_j\},d\right)}	\underset{v_k}{\overset{v_j}{\simeq}} M_{\Gamma_{G_k}\left(\{v_k\},d\right)}$.
\end{definition}

We are ready to state our generic result about the stabilization of coefficients phenomena.

\begin{theorem}[Power series stabilization] \label{thm:stabilization}
    Suppose that $\{(M_{G_j},v_j) \colon j\in \mathbb{N}\}$ is a convergent family of rooted PLUPs, then the coefficients of the power series of $R_{G_i}=\E_{G_i}(\Res{v_i})$ stabilize. In particular, $R_{G_i}(p) = R_{G_j}(p) + \bigO{p^{\min(i,j)+1}}$
\end{theorem}	
Note that for translation invariant graphs, this implies $R_{G_i}=\frac{1}{|G_i|} \E_{G_i}(\text{total steps})$ stabilizes.
\begin{proof}
    Let $d = \min(i,j)$, then
	\begin{align*}
        \E_{G_i}(\Res{v_i})
        &= \sum_{k\geq 0} k \cdot \P_{G_i}(\Res{v_i} = k)\\
        &= \sum_{k\geq 0} k \cdot \P_{G_i\cap \Gamma_{G_i}\left(v_i,d\right)}\(\Res{v_i} = k\) + \bigO{p^{d+1}} \tag{by Lemma~\ref{lemma:distancePower}}\\
        &= \sum_{k\geq 0} k \cdot \P_{G_j\cap \Gamma_{G_j}\left(v_j,d\right)}\(\Res{v_j} = k\) + \bigO{p^{d+1}} \tag*{$\left(M_{G_i \cap\Gamma_{G_i}\left(v_i,d\right)}
		\underset{v_j}{\overset{v_i}{\simeq}}
		M_{G_j\cap\Gamma_{G_j}\left(v_j,d\right)}\right)$}\\	
        &= \sum_{k\geq 0} k \cdot \P_{G_j}\(\Res{v_j} = k\) + \bigO{p^{d+1}}  \tag{by Lemma~\ref{lemma:distancePower}}.
	\end{align*} 
    In Claim~\ref{claim:absconvergence} in Appendix~\ref{apx:absconvergence}, we prove that these types of sums are absolutely convergent for small enough $p$. Therefore the equality holds when the left- and right-hand side are considered as a power series in $p$.
\end{proof}

\section{The discrete Bak-Sneppen process} \label{sec:numerics}

In Section~\ref{sec:stabilization} we introduced two quantities that exhibit a phase transition in the DBS process. We saw that the coefficients of their power series stabilize. In this section we will look at them in more detail.

\subsection{Notation}
We denote by $M_{G}$ the DBS process on the graph $G=(V,E)$. With a slight abuse of notation we also denote by $M_{G}$ the leaking transition matrix of this time-independent Markov Chain, where the row and column that correspond to the all-inactive configuration is set to zero. 
We will index vectors (and matrices) by sets $A\subseteq V$, where $A$ is the set of active vertices, as in Section~\ref{sec:plup}.
We will denote probability row vectors by $\rho \in \R^{2^n}$ so that $\rho \cdot M_G$ is the state of the system after one time step. Setting the all-inactive row and column to zero corresponds to the property that for every $A\subseteq V$ we have $(M_G)_{\emptyset,A} = (M_G)_{A,\emptyset} = 0$.
We will use notation $M_{(n)}$ for the process on the cycle of length $n$ and $M_{[n]}$ for the process on the chain (not periodic) of length $n$. In both case we identify vertices with  $V:=[n]=\{1,2,...,n\}$.

\subsection{Expected number of resamples per site} \label{sec:Rn}
The first quantity of interest is the expected number of steps per vertex to reach the all-inactive state. Consider the DBS process on the cycle of length $n$.
We start the process by letting each vertex be active with probability $p$ and inactive with probability $1-p$, independently for each vertex. Denote this initial state by $\rho^{(0)}$, so $\rho^{(0)}_A = p^{|A|}(1-p)^{n-|A|}$. Let $J$ be a vector with all entries equal to 1, except for the entry of the all-inactive state which is zero. Then $\rho^{(0)} \cdot M_{(n)}^k \cdot J^T$ is the probability that the all-inactive state has \emph{not} been reached in $k$ steps, starting from $\rho^{(0)}$.
Now define $R_{(n)}(p)$ as the expected number of steps per vertex, before reaching the all-inactive state:
\begin{align}
    R_{(n)}(p)  &= \frac{1}{n} \sum_{k=1}^{\infty} \P(\text{reach end in $k$ steps or more})  \nonumber\\
    &= \frac{1}{n} \sum_{k=1}^{\infty} \rho^{(0)} \cdot M_{(n)}^{k-1} \cdot J^T \label{eq:expectationsum}\\
    &= \frac{1}{n} \rho^{(0)} \cdot (\id - M_{(n)})^{-1} \cdot J^T \tag{by the geometric series} \nonumber\\     
    &= \frac{P_{(n)}(p)}{P'_{(n)}(p)}, \label{eq:rational}
\end{align}
where $P_{(n)},P'_{(n)}$ are polynomials as can be seen by using Cramer's rule for matrix inversion. Therefore we can conclude that $R_{(n)}(p)$ is a rational function.
For small $n$ we can compute the functions $R_{(n)}(p)$ by symbolic matrix inversion, which is how we obtained the coefficients in Table~\ref{tab:coeffs}. For $n\geq 9$ we computed the matrix inversion for rational values of $p$ exactly, and then computed the rational function using Thiele's interpolation formula.

\subsubsection{The power-series of $R_{(n)}(p)$}

\begin{figure}[h] 
	\begin{center}
		\includegraphics[draft=false,width=\textwidth]{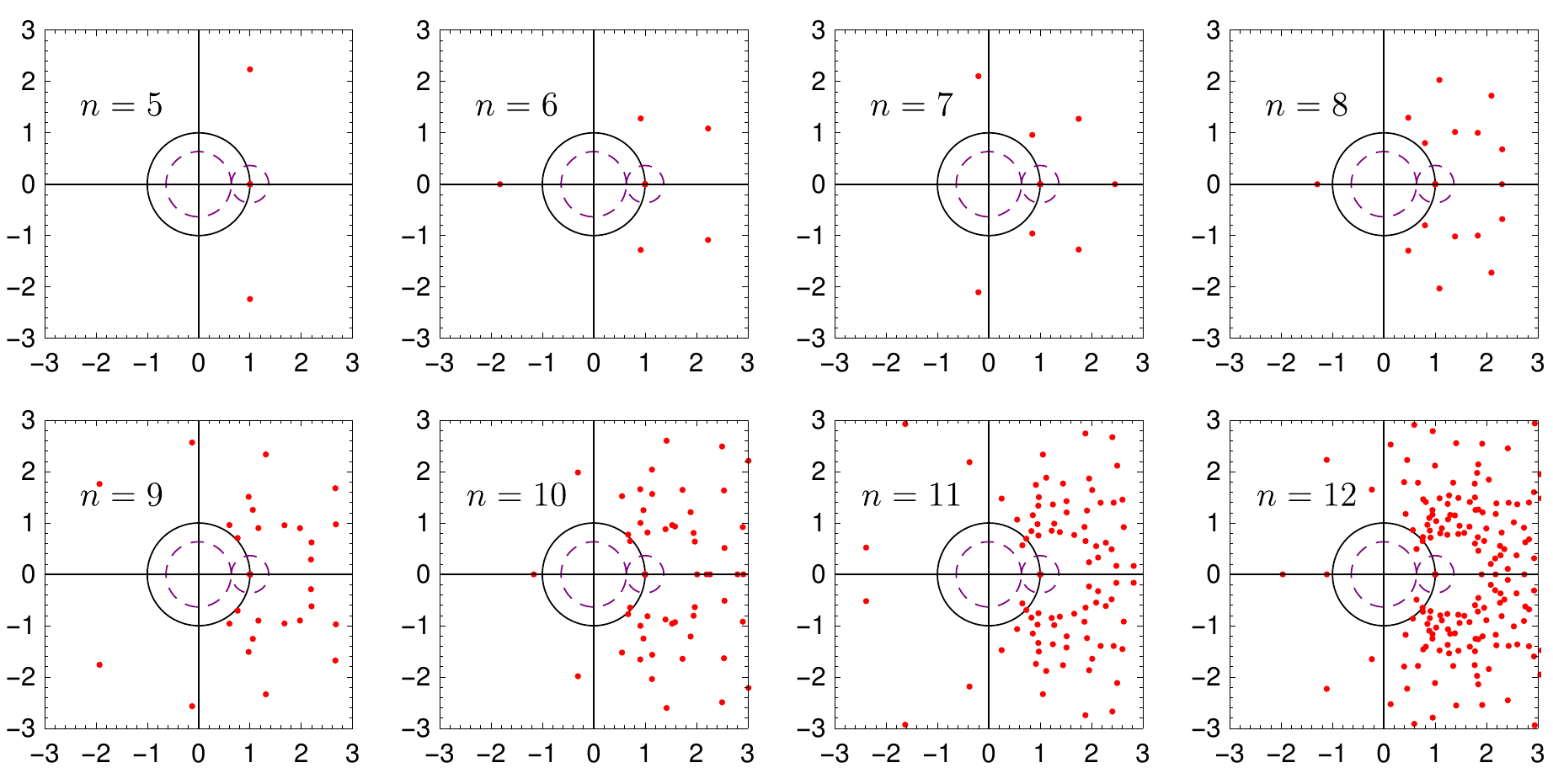}
   		\caption{\label{fig:plotpoles} Location of the poles of $R_{(n)}(p)$ in the complex plane for different $n$. The black circle is the complex unit circle and the dashed circles have radius $p_c$ around $p=0$ and $1-p_c$ around $p=1$. There is always a pole at $p=1$ because $R_{(n)}(1)$ is always infinite.}
	\end{center}
\end{figure}
As we have seen in the previous subsection, $R_{(n)}(p)$ is a rational function. Since a rational function is analytic, and $R_{(n)}(p)$ has no pole at $p=0$ (it actually takes value~$0$), we can write it as
\begin{equation}
R_{(n)}(p) = \sum_{k=0}^{\infty} a^{(n)}_k p^k, \label{eq:expectationseries}
\end{equation}
where the (non-zero) radius of convergence of the above power series equals the absolute value of the closest pole of $R_{(n)}(p)$ to $0$.
In order to get some intuition about the radius of convergence we plotted the location of the poles of $R_{(n)}(p)$ on the complex plane in Figure \ref{fig:plotpoles}. For $n=10$ there is a pole at a point with absolute value $\approx 0.9598$, hence $R_{(10)}(p)$ has a radius of convergence strictly smaller than $1$ even though the rational function $R_{(n)}(p)$ is well-defined for all $p\in[0,1)$. 

\begin{figure}[h]\centering
    \includegraphics[draft=false,width=0.43\textwidth]{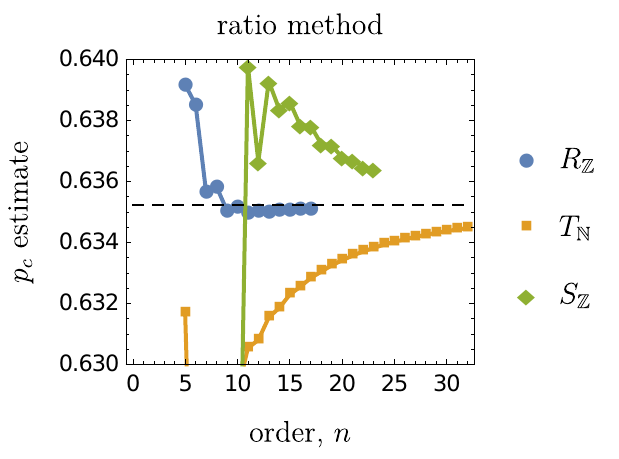}
    \;
    \includegraphics[draft=false,width=0.53\textwidth]{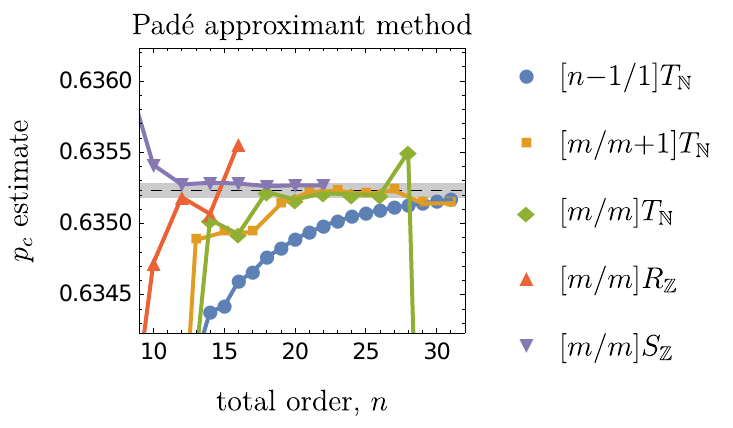}
    \caption{Estimates for $p_c$ based on the two methods. On the horizontal axis, $n$ is the number of power-series coefficients used for the estimate. The function $R_{\mathbb{Z}}$, $T_{\mathbb{N}}$ and $S_{\mathbb{Z}}$ are defined in the text below Conjecture~\ref{con:radiusconv}. The numbers $[m,m']$ (with $m+m'=n$) refer to the degree of the numerator and denominator respectively of the rational functions used in the Pad\'e approximant method. The gray shaded region shows our estimate $p_c = 0.63523 \pm 0.00005$.
    \label{fig:coeffs_conv_radius}}
\end{figure}

As was shown in Section~\ref{sec:stabilization}, Table~\ref{tab:coeffs}, the coefficients $a^{(n)}_k$ stabilize as $n$ grows. This is proven by Theorem~\ref{thm:stabilization}, since the family of DBS processes on the cycles, indexed by $n$, is a convergent familiy of PLUPs. The theorem only guarantees the stabilization for $n > 2k$ since going from a cycle of size $n$ to $n+1$ adds a vertex at a distance $n/2$ to any fixed vertex. In the table, however, we saw that the stabilization already holds for $n \geq k+1$. In Appendix~\ref{apx:preciseStabilization} we prove this more precise version of the stabilization that holds for cycles.
We define the `stabilized' coefficients $a^{(\infty)}_k = a^{(k+1)}_k$. We then define $R_{\mathbb{Z}}(p) = R_{(\infty)}(p) = \sum_{k=0}^\infty a^{(\infty)}_k p^k$ and make the following conjecture.
\begin{conjecture}[Radius of convergence] \label{con:radiusconv}
    The radius of convergence of $R_{(\infty)}(p)$ is equal to the critical probability $p_c$ of the DBS process.
\end{conjecture}

In Appendix~\ref{apx:absconvergence} we explain a method to compute coefficients of the $R_{(\infty)}(p)$ power series (see the text below Claim~\ref{claim:finitecontribution_Rn}).
As an application, we can apply known methods of series analysis. For example, Figure~\ref{fig:coeffs_conv_radius} shows estimates for $p_c$ using the ratio method and the Pad\'e approximant method. For details on these methods, see for example~\cite{Hunter1973}. The ratio method can be used to estimate the critical value when the singularity that determines the radius of convergence is at $p_c$, i.e. there are no other singularities closer to the origin, which is what we assume in Conjecture~\ref{con:radiusconv}.
The figure also shows estimates based on the power-series coefficients of the functions $T_{\mathbb{N}}$ and $S_{\mathbb{Z}}$.
The function $T_{\mathbb{N}}$ is the expected number of total steps on a chain with one end, with a single active vertex at that end as a starting state. This series is included because we can compute more terms for it.
The function $S_{\mathbb{Z}}$ is the probability of survival on the infinite line with a single active vertex as a starting state. This is a series in $q=1-p$ and it is included because other work studies the equivalent function for the contact process and this allows for comparison of critical exponents~\cite{Dickman1989}.
The Pad\'e approximant method suggests that the critical value is $p_c \approx 0.63523 \pm 0.00005$ and that the critical exponent for $S_{\mathbb{Z}}(q) \overset{q \uparrow q_c}{\sim} (q_c-q)^\beta$ is $\beta \approx 0.277$, which suggests that it is in the directed-percolation (DP) universality class alongside several variants of the contact process~\cite{Dickman1989,Inui1995,Tretyakov1997}.

\subsection{Reaching one end of the chain from the other}
Another quantity we considered in Section~\ref{sec:stabilization} is the probability of ever activating one end point of a finite chain, when we start the process with only a single active vertex at the other end.
Let us consider the length-$n$ chain, and suppose we start the DBS process with a single active vertex at site $1$. As in Equation~\eqref{eq:informalQseries}, we consider
\begin{align*}
    S_{[n]}(p) = \P(\BA{\{n\}} \mid \text{start }\{1\}).
\end{align*}
Note that in order to satisfy property~\ref{prop:indepInit} of the PLUP definition, the initial state needs to be $\{1\}$ with probability $p$ and $\emptyset$ with probability $1-p$. To get the above definition of $S_{[n]}(p)$ with a deterministic starting state one can then simply divide by $p$. 
The power-series coefficients of $S_{[n]}(p)$ stabilize, which follows from Lemma~\ref{lemma:distancePower} by letting $X=\{n\}$ and $Y=\{1\}$. However, as suggested by Figure \ref{fig:plotreachend}, the limiting power series around $p=0$ will become the zero function and it is therefore not so interesting. Instead, we can take the power series centered around $p=1$ and it turns out that also there the coefficients stabilize. We prove this below. Define $q=1-p$.

\begin{figure}[ht]
	\begin{center}
		\includegraphics[draft=false,width=\textwidth]{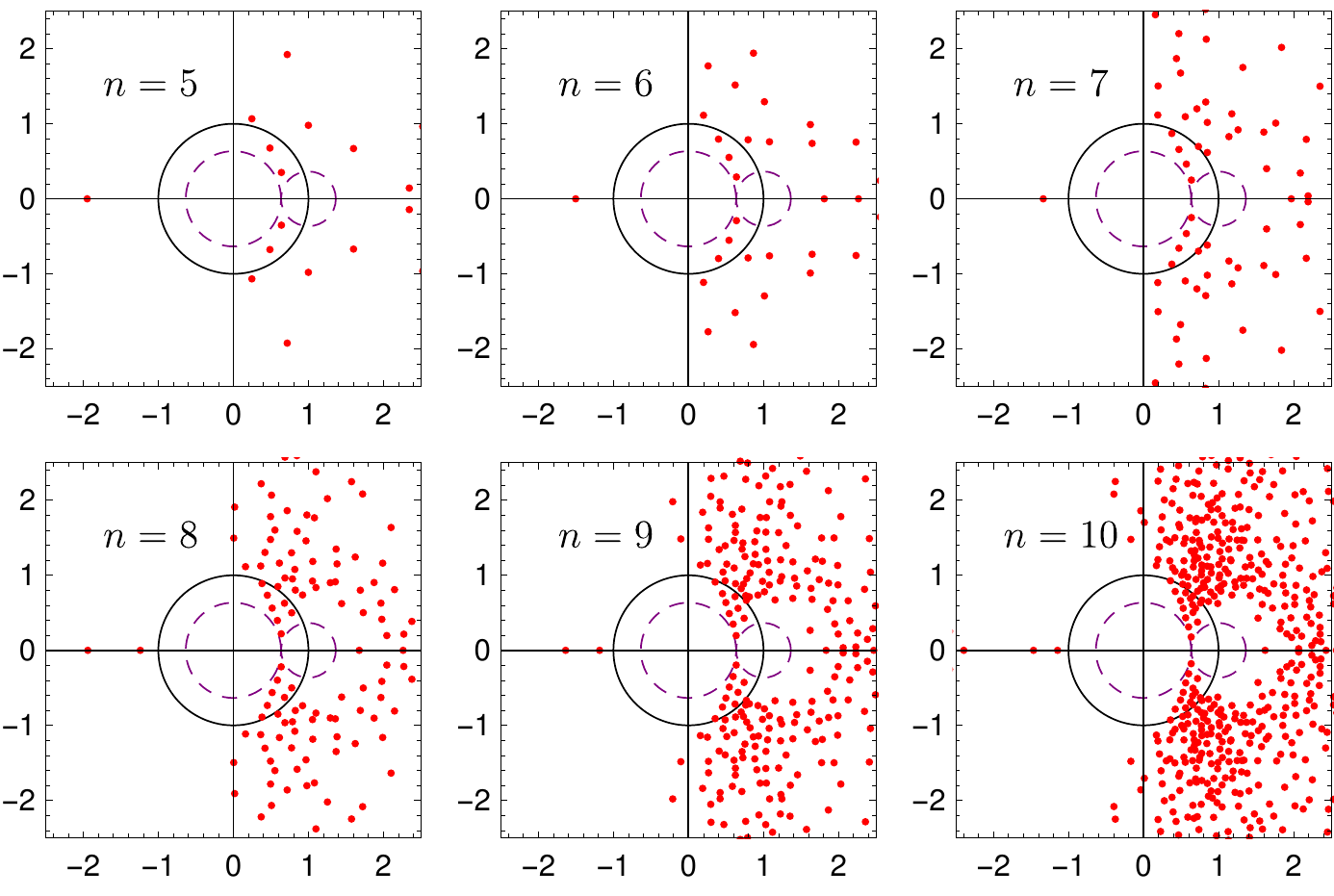}
        \caption{\label{fig:Qplotpoles} Location of the poles of $S_{[n]}$ \textbf{as a function of $p$} in the complex plane for different $n$. The black circle is the complex unit circle and the dashed circles have radius $p_c$ around $p=0$ and $1-p_c$ around $p=1$.}
	\end{center}
\end{figure}
\newcommand{\irchi}[2]{\raisebox{\depth}{$#1\chi$}} 
\DeclareRobustCommand{\rchi}{{\mathpalette\irchi\relax}}
Similarly to what we did for $R_{(n)}(p)$ we can write $S_{[n]}(q)$ using a matrix inverse. We will start the process in the (deterministic) state with a single active vertex at location 1, denoted by the probability vector $\delta_{\{1\}}$. Define $\mathcal{A}_n = \{ A \subseteq [n] \mid n \in A \}$, the set of all states where vertex $n$ is active. Let $M_{[n]}$ be the transition matrix for the DBS process on the chain of length $n$. Define the matrix $\tilde{M}_{[n]}$ as $M_{[n]}$ but with some entries set to zero. Set the row and column of the all-inactive state $\emptyset$ to zero, $(\tilde{M}_{[n]})_{A,\emptyset} = (\tilde{M}_{[n]})_{\emptyset,A} = 0$ for all $A\subseteq [n]$. Furthermore set all rows $A\in \mathcal{A}_n$ to zero: $(\tilde{M}_{[n]})_{A,A'} = 0$ for all $A'\subseteq [n]$. That is, whenever vertex $n$ is active there is no outgoing transition. Denote by $\rchi_{\!\mathcal{A}_n}$ the vector that is 1 for all $A\in\mathcal{A}_n$ and zero everywhere else. We have
\begin{align}
    S_{[n]}(q) &= \P( \text{vertex $n$ becomes active} ) \nonumber\\
             &= \sum_{k\geq 0} \P( \text{vertex $n$ activates for the first time at step $k$} ) \nonumber\\
             &= \sum_{k\geq 0} \delta_{\{1\}} \cdot \tilde{M}_{[n]}^k \cdot \rchi_{\!\mathcal{A}_n} \nonumber\\
             &= \delta_{\{1\}} \cdot (\id - \tilde{M}_{[n]})^{-1} \cdot \rchi_{\!\mathcal{A}_n} \tag{by the geometric series} \nonumber\\
             &= \sum_{k\geq 0} b^{[n]}_k q^k \label{eq:Qseries}
\end{align}
With the same argument as before we see that $S_{[n]}$ must be a fraction of two polynomials in $p$ (and also in $q$).
The poles of $S_{[n]}$ are shown in Figure \ref{fig:Qplotpoles} where $S_{[n]}$ is considered a function of $p$ to be comparable with $R_{(n)}(p)$. The coefficients $b^{[n]}_k$ of the $q$ power series are shown in Table~\ref{tab:Qcoeffs}.

\begin{claim}
    The coefficients $b^{[n]}_k$ of the power series of $S_{[n]}(q)$ in Equation~\eqref{eq:Qseries} stabilize.
\end{claim}
\begin{proof}
    Let $\RI{\{n\}}$ and its complement $\BA{\{n\}}$ be as defined in Definition \ref{def:events}.
    In the following we assume that the starting state is $\{1\}$ with probability $p$ and $\emptyset$ with probability $1-p$, so the process is a PLUP. We have $S_{[n]}(p)= \frac{1}{p} \cdot \P(\BA{n})$, since $S_{[n]}(p)$ has a deterministic starting state.
    By Lemma \ref{lemma:splitting} we have $\P_{[n]}(\RI{\{n-1\}}) = \P_{[n-1]}(\RI{\{n-1\}})$.
    Consider $1-p S_{[n]}$, i.e. the probability that the $n$-th vertex is \emph{not} activated. We have
    \begin{align*}
        1 - p S_{[n]} &= \P_{[n]}( \RI{\{n\}} ) \tag{definition of $S_{[n]}$}\\
        &= \P_{[n]}(\RI{\{n-1\}} \cap \RI{\{n\}}) + \P_{[n]}(\BA{\{n-1\}} \cap \RI{\{n\}}) \tag{partition of events}\\
        &= \P_{[n-1]}(\RI{\{n-1\}}) + \P_{[n]}(\BA{\{n-1\}} \cap \RI{\{n\}}) \tag{Lemma \ref{lemma:splitting}}\\
        &= 1 - p S_{[n-1]} + \P_{[n]}(\BA{\{n-1\}} \cap \RI{\{n\}}) .
    \end{align*}
    Note that for the event $(\BA{\{n-1\}} \cap \RI{\{n\}})$ to hold, all vertices $1,...,n-1$ must have been active. Since the process terminates with probability 1, this means all those vertices must also have been deactivated at least once. In the DBS process a deactivation is $\bigO{q}$, so every terminating path of the Markov Chain that is in this set has a factor of at least $q^{n-1}$ associated to it, hence $\P_{[n]}(\BA{\{n-1\}} \cap \RI{\{n\}}) = \bigO{q^{n-1}}$. Here we use the absolute convergence of certain power series in $q$, which we prove in Claim~\ref{claim:absqconvergence} in Appendix~\ref{apx:absqconvergence}. We see that $S_{[n]}(q) - S_{[n-1]}(q) = \bigO{q^{n-1}}$ so the coefficients stabilize.
\end{proof}

\bibliographystyle{alphaUrlePrint}
\bibliography{main.bib}

\appendix

\newpage

\section{Decay of correlations} \label{apx:correlations}

Recall Corollary~\ref{cor:decay}.

\decay*

\begin{proof}
	First observe that if $d(X,Y)=\infty$, it means that either $X$ and $Y$ are in different connected components of $G$, or one of them is the empty set, therefore $L_X$ and $L_Y$ are independent events, so the statement holds.
	
	Note that due to Property~\ref{prop:indepInit} the only path which has a non-zero constant term is the trivial path, when every vertex is initialized as inactive, thus the constant term of the probability of any local event is either $0$ or $1$. Also the constant term of $\P_{G}(L_X\cap L_Y)$ is $1$ if and only if the constant terms of both $\P_{G}(L_X)$ and $\P_{G}(L_Y)$ are $1$, which concludes the $d(X,Y)=0$ case.  
    
    Note that by De Morgan's law, \eqref{eq:BACovariance} is equivalent with 
	\begin{equation}\label{eq:RICovariance}
		\P_{G}(\RI{X}\cap \RI{Y})-\P_{G}(\RI{X})\P_{G}(\RI{Y})=\bigO{p^{d(X,Y)+1}}.
	\end{equation}
	Now we proceed by induction on $d(X,Y)$. Assume \eqref{eq:LCovariance}-\eqref{eq:BACovariance} hold for $d(X,Y)=d-1$. We will prove the statement for $d(X,Y)=d$. 
    We apply a similar idea as in the proof of Lemma~\ref{lemma:distancePower}. Define
    \begin{align*}
        L_X^i     &:= L_X \cap \RI{\overline{\partial}(X,i)} \cap \bigcap_{j\in[i-1]} \BA{\overline{\partial}(X,j)} \qquad\text{for}\quad i \in [d-1],\\
        L_X^{d} &:= L_X \cap \bigcap_{j\in [d-1]} \BA{\overline{\partial}(X,j)} .
    \end{align*}
    When $L_X^i$ holds, everything at distance $i$ remains inactive, but for all distances $j \leq i-1$ there exist vertices that become active at that distance.
    These events form a partition $L_X=\dot\bigcup_{i\in [d]}L_X^{i}$, and similarly for $L_Y^i$. Below we depict $L_X^{i} \cap L_Y^j$ graphically.
    \begin{center}
        \includegraphics[draft=false]{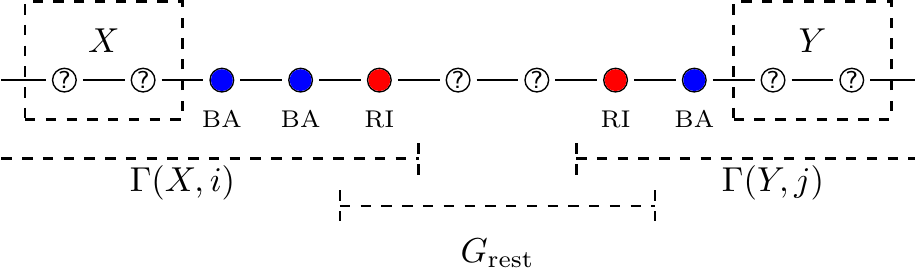}
    \end{center}
	We will show the inductive step for both \eqref{eq:LCovariance}-\eqref{eq:BACovariance} at the same time, for which we introduce a number $c$ such that $c=1$ if $L_X=\BA{X}$ and $L_Y=\BA{Y}$, and $c=-1$ otherwise.
    By Lemma~\ref{lemma:activationpower}
    \begin{align} \label{eq:eventpowers}
        \P(L_X^i \cap L_Y^j) = \bigO{p^{i+j-1 + c}} \qquad\text{and}\qquad
        \P(L_X^i) \cdot \P(L_Y^j) = \bigO{p^{i+j-1 + c}} ,
    \end{align}
    for any graph on which the events are defined. Since the events form a partition, we have
    \begin{align*}
        \P_G(L_X \cap L_Y) = \sum_{i,j\in[d]}\P_G(L_X^i \cap L_Y^j) \qquad\text{and}\qquad
        \P_G(L_X) \cdot \P_G(L_Y) = \sum_{i,j\in[d]}\P_G(L_X^i) \cdot \P_G(L_Y^j),
    \end{align*}
    so it is sufficient to prove the statement for each $i,j$ separately, i.e. we want to show
    \begin{align*}
        \P_G(L_X^i \cap L_Y^j) - \P_G(L_X^i) \P_G(L_Y^j) = \bigO{p^{d+c}} .
    \end{align*}
    When $i+j-1 \geq d$ then its trivial by~\eqref{eq:eventpowers}. Now fix $i,j$ such that $i+j \leq d$ and define $G^{i,j}_\mathrm{rest} := G \setminus \(\;\Gamma(X,i-1)\cup \Gamma(Y,i-1)\;\)$, as indicated in the diagram. The $\RI{..}$ events split the graph in three parts, so we have
    \begin{align*}
		\!\!\!\!\P_{G}(L_X^i \cap L_Y^j)
        &= \P_{\Gamma(X,i)}(L_X^i \mid \II{\overline{\partial}(X,i)}) \cdot \P_{\Gamma(Y,j)}(L_Y^j \mid \II{\overline{\partial}(Y,j)}) \cdot \P_{G^{i,j}_\mathrm{rest}}(\RI{\overline{\partial}(X,i)} \cap \RI{\overline{\partial}(Y,j)})
        \tag{using Lemma~\ref{lemma:splitting} twice} \\
        &= \P_{\Gamma(X,i)}(L_X^i \mid \II{\overline{\partial}(X,i)}) \cdot \P_{\Gamma(Y,j)}(L_Y^j \mid \II{\overline{\partial}(Y,j)}) \\
        & \qquad \cdot \left[ \P_{G^{i,j}_\mathrm{rest}}(\RI{\overline{\partial}(X,i)}) \cdot \P_{G^{i,j}_\mathrm{rest}}(\RI{\overline{\partial}(Y,j)}) + \bigO{p^{(d-i-j)+1}} \right]
        \tag{by induction of~\eqref{eq:RICovariance}} \\
        &= \P_{\Gamma(X,i)}(L_X^i \mid \II{\overline{\partial}(X,i)}) \cdot \P_{\Gamma(Y,j)}(L_Y^j \mid \II{\overline{\partial}(Y,j)}) \\
        & \qquad
        \cdot \P_{G^{i,j}_\mathrm{rest}}(\RI{\overline{\partial}(X,i)})
        \cdot \P_{G^{i,j}_\mathrm{rest}}(\RI{\overline{\partial}(Y,j)}) + \bigO{p^{d+c}}
        \tag{by~\eqref{eq:eventpowers}} \\
        &= \P_{\Gamma(X,i)}(L_X^i \mid \II{\overline{\partial}(X,i)}) \cdot \P_{\Gamma(Y,j)}(L_Y^j \mid \II{\overline{\partial}(Y,j)}) \\
        & \qquad
        \cdot \P_{G\setminus\Gamma(X,i-1)}(\RI{\overline{\partial}(X,i)})
        \cdot \P_{G\setminus\Gamma(Y,j-1)}(\RI{\overline{\partial}(Y,j)}) + \bigO{p^{d+c}}
        \tag{by Lemma~\ref{lemma:distancePower} and~\eqref{eq:eventpowers}} \\
        &= \P_{G}(L_X^i) \cdot \P_{G}(L_Y^j) + \bigO{p^{d+c}}
        \tag{using Lemma~\ref{lemma:splitting} twice}
    \end{align*}
\end{proof}

\section{Absolute convergence} \label{apx:absconvergence}

Recall Lemma~\ref{lemma:activationpower}.

\activationpower*

The proof requires the regrouping of terms in an infinite sum. In this section we prove the absolute convergence of certain series that allows for this regrouping.
We start with some notation.

\newcommand{\paths}[1]{\textsc{paths}_{#1}}
\newcommand{\tpaths}[1]{\textsc{tpaths}_{#1}}
\newcommand{\mindeg}{\mathrm{mindeg}}
\newcommand{\maxdeg}{\mathrm{maxdeg}}
\newcommand{\absp}[1]{\left\lVert #1 \right\rVert_{\mathrm{abs}}}

\begin{definition}[Paths] \label{def:paths1}
    Define a \emph{path} of length $k$ as an initialization and sequence of $k$ updates, where we only count steps in which an active vertex was selected. We write a path $\xi$ as
\begin{align*}
    \xi=\left( (\text{initialize to }A_0), (v_1, R_1), (v_2, R_2), ..., (v_k, R_k) \right).
\end{align*}
    Here $v_i$ denotes the vertex that was selected in the $i$-th step and $R_i\subseteq \Gamma(v_i)$ is the result of the corresponding update that happened afterwards. After $t$ steps, the state of the process is $A_t = (A_{t-1} \setminus \Gamma(v_t) ) \cup R_t$. We say a path is terminating if $A_k = \emptyset$.
    Denote by $\paths{A,k}$ the set of all paths $\xi$ that initialize to $A$ and have length $k$.
    Denote by $\tpaths{A,k}$ the set of all \emph{terminating} paths that initialize to $A$ and have length $k$.
\end{definition}
We have $\tpaths{A,k}\subseteq\paths{A,k}$. Any local event is a collection of terminating paths\footnote{Up to the zero probability event of non-termination.}, and
\begin{align} \label{eq:sumtpaths}
    \sum_{k=0}^{\infty} \sum_{A\subseteq [n]} \sum_{\xi\in\tpaths{A,k}} \P(\xi) = 1.
\end{align}
For a general PLUP we have $\P(\xi) = \P(A_0) \P( (v_1,R_1) \mid A_0) \P( (v_2,R_2) \mid A_1) \cdots \P( (v_k,R_k) \mid A_{k-1})$ where the polynomial $\P(A_0)$ is the probability of starting in state $A_0$ and $\P( (v_t,R_t) \mid A_{t-1})$ are polynomials satisfying property \ref{prop:locUpdateRule} of the PLUP definition.
For the DBS process these polynomials take the specific form $\P(\xi) = \P(A_0) Z_\xi p^{|R_1| + ... + |R_k|} (1-p)^{3k - |R_1| + ... + |R_k|}$ where $Z_\xi$ is some $p$-independent factor. 
\begin{definition}[Polynomials]
    Let $Q(p) = a_m p^m + a_{m+1} p^{m+1} + ... + a_M p^M$ be a polynomial where $a_m\neq0$ and $a_M\neq 0$.
    Define $\mindeg(Q(p)) = m$, $\maxdeg(Q(p)) = M$ and define by $\absp{Q}$ the polynomial obtained by taking the absolute values of the coefficients:
    \begin{align*}
        \absp{Q}(p) = |a_m| p^m + |a_{m+1}| p^{m+1} + ... + |a_M| p^M.
    \end{align*}
\end{definition}
By the triangle inequality $\absp{f\cdot g}(p) \leq \absp{f}(p) \cdot \absp{g}(p)$ for any polynomials $f,g$.
\begin{claim}\label{claim:finitecontribution}
    The polynomials $\P(\paths{A,k})$, $\P(\tpaths{A,k})$ and $\P(\xi)$ for any $\xi\in\paths{A,k}$ all satisfy
    \begin{align*}
        \mindeg(\cdot) \geq c \cdot (k - |A|) + \mindeg(\P(A))
        \; , \qquad
        \maxdeg(\cdot) \leq c' \cdot k + \maxdeg(\P(A)).
    \end{align*}
    Here $0 < c < c'$ are constants depending on the particular process and $\P(A)$ is the probability of starting in state $A$ (a polynomial).
\end{claim}
\begin{proof}
    Note that
    \begin{align*}
        \P(\paths{A,k})
        = \sum_{\xi \in \paths{A,k}} \P(\xi)
    \end{align*}
    is a sum over finitely many polynomials. It is sufficient to prove the statement for each $\xi$ and it then follows for the sum. Since $\tpaths{A,k}\subseteq\paths{A,k}$ it also follows for $\P(\tpaths{A,k})$. Let $\xi$ be a path as described in Definition~\ref{def:paths1}. As stated in the text below Definition~\ref{def:paths1} we have $\P(\xi) = \P(A_0) \P( (v_1,R_1) \mid A_0) \P( (v_2,R_2) \mid A_1) \cdots \P( (v_k,R_k) \mid A_{k-1})$ where $A_t\subseteq[n]$ is the state after $t$ steps and $A_0 = A$. Let $c'$ be the degree of the highest order term of any possible local-update step of this process (finitely many possibilities) than $\maxdeg(\P(\xi)) \leq c' \cdot k + \maxdeg(\P(A))$.

    Note $\mindeg(\P(\xi)) = \mindeg(\P(A)) + \mindeg(\P((v_1,R_1)\mid A_0)) + \cdots + \mindeg(\P((v_k,R_k)\mid A_{k-1}))$.
    If $|A_{t}| - |A_{t-1}| \geq 0$ then either $A_t = A_{t-1}$ or $|A_t \setminus A_{t-1}| > 0$.
    By property \ref{prop:locUpdateRule} of the PLUP definition we therefore have that $|A_{t}| - |A_{t-1}| \geq 0$ implies $\mindeg( \P((v_t,R_t)\mid A_{t-1}) ) \geq 1$. Furthermore, $|A_{t}| - |A_{t-1}| \leq d_\mathrm{max}$ where $d_\mathrm{max}$ is the maximum degree of the vertices in $G$.
    Therefore we have
    \begin{align*}
        \mindeg( \P((v_t,R_t)\mid A_{t-1}) ) \geq \frac{1}{d_\mathrm{max}+1}\( 1 + |A_{t}| - |A_{t-1}| \)
    \end{align*}
    Summing both sides over $t$ we obtain
    $\mindeg(\P(\xi)) - \mindeg(\P(A)) \geq \frac{1}{d_\mathrm{max}+1}( k + |A_k| - |A| )$.
    This proves the claim with $c = \frac{1}{d_\mathrm{max}+1}$.
\end{proof}

For the DBS process, in the context of Section~\ref{sec:Rn}, we can slightly refine the above claim.
\begin{claim} \label{claim:finitecontribution_Rn}
    Let $\rho^{(0)}$, $M_{(n)}$ and $J$ be as defined in Section~\ref{sec:Rn}. The polynomial $\rho^{(0)}\cdot M_{(n)}^k \cdot J^T = \sum_{A\subseteq [n]} \P(\paths{A,k})$ in $p$ has lowest-order term at least $p^{k}$ and highest-order term at most $p^{n+3k}$. \tnote{technically the lowest order term is $p^{k+1}$ because $J_{\emptyset} = 0$.}
\end{claim}
\begin{proof}
    We repeat the proof of Claim~\ref{claim:finitecontribution}. Note that $\mindeg( \P((v_t,R_t)\mid A_{t-1}) ) \geq 1 + |A_{t}| - |A_{t-1}|$ for the DBS process, so $c=1$. For DBS, $c' = 3$ which is the maximum degree of the local update rule ($p^3$ occurs when all three resampled vertices become active). The claim then follows by noting that $\P(A) = p^{|A|}(1-p)^{n-|A|}$ in the starting state $\rho^{(0)}$.
\end{proof}
\pagebreak[1]

This claim is convenient for the computation of the $R_{(n)}(p)$ power series.
It implies that the term $p^j$ is only present in those polynomials $\rho^{(0)} \cdot M_{(n)}^k \cdot J^T$ for which $\lceil\frac{j-n}{3}\rceil \leq k \leq j$. To compute the power-series coefficient $a^{(n)}_j$ it is sufficient to consider this finite set of polynomials. In other words, in order to compute $R_{(n)}(p)$ up to $k$-th order in $p$, it suffices to consider only the first $k$ steps of the DBS process.
We use this observation to compute the coefficients of the $n\geq 18$ series by computing matrix powers symbolically in $p$, see Table~\ref{tab:coeffs}.

\begin{claim}\label{claim:absconvergence}
    There is a constant $\delta>0$ such that, for any polynomial $f(k)$, the following series is absolutely convergent for $p \in [0,\delta]$:
\begin{align*}
    \sum_{k=0}^{\infty} \sum_{A\subseteq [n]} \sum_{\xi\in\paths{A,k}} f(k) \absp{\P(\xi)} < \infty .
\end{align*}
\end{claim}
Note that the sum is over all paths, not only the terminating ones.

\begin{proof}
    We have $\P(\xi) = \P(A_0) \P( (v_1,R_1) \mid A_0) \P( (v_2,R_2) \mid A_1) \cdots \P( (v_k,R_k) \mid A_{k-1})$.
    The polynomials $P_t := \P((v_t,R_t)\mid A_{t-1})$ come from a finite set of polynomials: for each vertex $v$ there are at most $2^{|\Gamma(v)|}$ possible updates and there are at most $n$ vertices.
    Therefore there is a constant $C$ such that for all these polynomials
    \begin{align*}
        \absp{P_t} \leq C \; p^{\mindeg(P_t)} .
    \end{align*}
    By Claim~\ref{claim:finitecontribution} there is a $c$ such that
    \begin{align*}
        \absp{\P(\xi)} \leq \absp{\P(A_0)} \absp{P_1} \cdots \absp{P_k} &\leq \absp{\P(A_0)} C^k \; p^{\mindeg(P_1) +\cdots \mindeg(P_k)} \\
        &\leq \absp{\P(A)} C^k p^{c\cdot(k-|A|)} .
    \end{align*}
    There are at most $(2^{d_\mathrm{max}}n)^k$ paths of length $k$ for a fixed starting state so we have
    \begin{align*}
        \sum_{k=0}^{\infty} \sum_{A\subseteq [n]} \sum_{\xi\in\paths{A,k}} f(k) \absp{\P(\xi)}
        &\leq \sum_{k=0}^{\infty} \sum_{A\subseteq [n]} f(k) \absp{\P(A)} (2^{d_\mathrm{max}}n)^k C^k p^{c(k-|A|)}
    \end{align*}
    Since there are finitely many ($2^n$) starting states $A$, the whole expression is absolutely convergent for $p < \(2^{d_\mathrm{max}}n C\)^{-1/c}$.
\end{proof}
Since any local event is a subset of the set of all terminating paths, the powerseries for any event is also absolutely convergent. This yields the following corollaries.
\begin{corollary} \label{cor:pathsevents}
    Let $E$ be an event such that $\P(\xi) = \bigO{p^k}$ for all paths $\xi\in E$. Then $\P(E) = \bigO{p^k}$.
\end{corollary}
\begin{proof}
    For $p \in [0,\delta]$ we have $\P(E) = \sum_{j=k}^\infty a_j p^j$ by Claim~\ref{claim:absconvergence}. By uniqueness of power series, this equality holds for all $p$ up to the radius of convergence.
\end{proof}
\begin{corollary}
    Let $A_0\subseteq [n]$ be any state and let $E$ be an event such that for all paths $\xi\in E$ we have $\P(\xi \mid \text{start }A_0) = \bigO{p^k}$. Then $\P(E \mid \text{start }A_0) = \bigO{p^k}$.
\end{corollary}
\begin{proof}
    The proofs of Claim \ref{claim:absconvergence} and Corollary \ref{cor:pathsevents} also hold when everything is conditioned on a fixed starting state $A_0\subseteq [n]$. The sum over $A$ and the factors $\absp{\P(A)}$ are simply removed from the equations.
\end{proof}

\section{Absolute convergence for the $q$ power series} \label{apx:absqconvergence}

We now turn our attention to the $S_{[n]}(q)$ series defined in Equation \eqref{eq:Qseries}. This process starts with a single active vertex at position $1$, i.e. $A=\{1\}$, and we look at the probability that vertex $n$ is never activated, $\P(\RI{\{n\}} \mid \text{start } A)$, as a function of $q=1-p$. To prove the absolute convergence of such series for general PLUPs we require more properties than the definition of PLUP that we have. We now consider the update polynomials as a function of $q=1-p$. The update rule for a single time step should satisfy the following two properties.
\begin{itemize}
    \item When $q=0$ then the probability that an active vertex becomes inactive is zero.\\
        This implies that any inactivation is $\bigO{q}$.
    \item There is a $c>0$ such that if $q=0$ then: for all vertices $v$ with an active neighbor, the probability of activating $v$ when that neighbor is selected is at least $c$.
\end{itemize}
These properties are satisfied by the CP and DBS process.

\begin{claim} \label{claim:absqconvergence}
    Consider a PLUP that satisfies the above two properties.
    Let $X \subset V$ be any subset of vertices such that its boundary $B=\bar{\partial}(X;1)$ is not empty. Let the starting state be $A \subseteq X$. Then the following series converges for small enough $q$
    \begin{align*}
        \sum_{k\geq 0} \sum_{\xi \in \tpaths{A,k} \cap \RI{B}} \absp{\P(\xi)}(q) < \infty .
    \end{align*}
    Therefore one can regroup terms in the series $\P(\RI{B}) = \sum_{k\geq 0}\sum_{\xi\in\tpaths{A,k}\cap\RI{B}} \P(\xi)(q)$.
\end{claim}

\begin{proof}
    Let $Q_{i,j}(q)$ be the $j$-th transition polynomial for vertex $i$, so that $\sum_j Q_{i,j}(q) = 1$. (In case of the DBS process these are of the form $q^a(1-q)^b$ and independent of which vertex is being resampled.) Since this holds for $q=0$, the constant terms of $Q_{i,j}(q)$ are non-negative and sum to 1. Hence we have $\sum_j \absp{Q_{i,j}}(q) = 1 + \bigO{q}$.
    Define new normalized functions $\tilde{Q}_{i,j}(q) = \frac{ \absp{Q_{i,j}}(q) }{ \(\sum_j \absp{Q_{i,j}}(q) \)}$ and consider the same process but with the transition polynomials $Q_{i,j}(q)$ replaced by the functions $\tilde{Q}_{i,j}(q)$. Whenever $Q_{i,j}(q) = \bigO{q}$ then also $\tilde{Q}_{i,j}(q) = \bigO{q}$.

    By the second additional property there is a $c$ such that any neighbor of an active vertex can be activated with probability at least $c$ when $q=0$, so the constant term in the corresponding transition polynomials is at least $c$. When taking the absolute polynomials $\absp{Q_{i,j}}(q)$ this also holds for non zero $q$. The functions $\tilde{Q}_{i,j}(q)$ are then at least $(c + \bigO{q} ) / (1 + \bigO{q})$ so there is a $q_0>0$ such that $\tilde{Q}_{i,j}(q) \geq c' = c/2$ for all $0\leq q \leq q_0$.
    
    Define the following random variables. Let $I^B_i \in \{0,1\}$ be 1 if any active vertex got inactivated in step $i$ \textit{or} if the process has terminated \textit{or} if a vertex in $B$ has been active at some point before step $i$. Let $G^B_i \in \{0,1\}$ be 1 if any inactive vertex got activated in step $i$ \text{or} if the process has terminated \textit{or} if a vertex in $B$ has been active at some point before step $i$.
    Let $I^B_{\leq k}=\sum_{i=1}^k I^B_i$ and $G^B_{\leq k} = \sum_{i=1}^k G^B_i$. We always have $G^B_{\leq k} \leq |X| + I^B_{\leq k}$ since after $|X|$ activations any other activation requires an inactivation first.

    Define $s = \frac{\min_{v\in V} w_v}{\sum_{v\in V} w_v}$ where the $w_v$ are the selection weights of the PLUP. (For the DBS process $s=\frac{1}{n}$). This is unchanged in the process with the $\tilde{Q}_{i,j}(q)$ transition functions.
    Now we show that
\begin{claim} \label{claim:coupling}
    The random variable $G^B_{\leq k}$ satisfies
    \begin{align*}
        \tilde{\P}\left(G^B_{\leq k} \leq \frac{s\cdot c'}{2} k\right) \leq \exp\(- s c' k / 8\) .
    \end{align*}
\end{claim}
\begin{proof}
    By definition, if there are no active or no inactive vertices in step $i$ then $G^B_i =1$. If there are, then an active vertex with inactive neighbor is picked with probability at least $s$ (this holds for both the original process and the process with $\tilde{Q}_{i,j}(q)$ transition functions) and the neighbor is then activated with probability at least $c'$. Therefore we have $\tilde{\P}(G^B_i = 1) \geq s\cdot c' > 0$.

    Define $k$ i.i.d. Bernoulli variables $C_i$ with success probability $s\cdot c'$ and $C=\sum_{i=1}^k C_i$. The expectation of $C$ is $\E(C) = s c' k$ and using the Chernoff bound we can bound the probability that $C$ deviates far from its mean:
    \begin{align*}
        \P\(C \leq \frac{1}{2} \E(C) \) \leq e^{- \frac{1}{4} \E(C)}
    \end{align*}
    We use a coupling argument to compare the $G^B_i$ variables with the $C_i$'s.
    Let $U_i$ be i.i.d. uniform $[0,1]$ variables.  Define $C_i$ to be $1$ if $U_i < s c'$ so the $C_i$'s are indeed i.i.d. Bernoulli variables with the correct distribution.

    The probability $\P(G^B_i = 1)$ depends on the history of the process. Run the process and in each step of the process, first compute the true probability $p_i = \P(G^B_i = 1 \mid \text{history})$, so $p_i > s c'$. Now instead of running the process normally, define $G^B_i=1$ if and only if $U_i < p_i$. Then continue the process conditioned on the value of $G^B_i$ so it follows the normal distribution.
    This way, the $G^B_i$ variables come from the correct distribution but they are coupled to the $C_i$'s. We see $C_i =1$ implies $G^B_i=1$ so $G^B_{\leq k} \geq C$ and therefore the Chernoff bound applies to $G^B_{\leq k}$ as well.
\end{proof}
    Now we continue the proof of Claim~\ref{claim:absqconvergence}.

    By the first property every inactivation has an update step that is $\bigO{q}$ (even with the $\tilde{Q}_{i,j}$ functions). Every path of length $k$ has $I^B_{\leq k}$ inactivations so has probability $\bigO{q}^{I^B_{\leq k}}$.
    Let $T_k$ be the event that the process takes exactly $k$ steps. For the ``tilde process'' we have
    \begin{align*}
        \tilde{\P}( \RI{B} \cap T_k) &= \tilde{\P}(\RI{B} \cap T_k \cap (G^B_{\leq k} \leq \frac{s\cdot c'}{2}k)) + \tilde{\P}(\RI{B} \cap T_k \cap (G^B_{\leq k} > \frac{s\cdot c'}{2}k)) \\
                             &\leq \exp(-\frac{s\cdot c'}{8} k) + \bigO{q}^{\frac{s\cdot c'}{2}k-|X|}
    \end{align*}
    Note that the constants in all big-O parts are independent of the number of steps.
    We have
    \begin{align*}
              \sum_{k\geq 0} \sum_{\xi \in \tpaths{A,k} \cap \RI{B}} \absp{\P(\xi)}(q)
        &\leq \sum_{k\geq 0} (1 + \bigO{q})^k \sum_{\xi \in \tpaths{A,k} \cap \RI{B}} \tilde{\P}(\xi)(q) \\
        &\leq \sum_{k\geq 0} (1 + \bigO{q})^k \( \exp(-\frac{s\cdot c'}{8} k) + \bigO{q}^{\frac{s\cdot c'}{2}k-|X|} \) .
    \end{align*}
    This is convergent for small enough $q$.
\end{proof}

\section{Proving that $a_k^{(k+1)}=a_k^{(n)}$ for all $n>k$}\label{apx:preciseStabilization}
In the main text we looked at $R_{(n)}(p)$ and saw that the coefficients of this series stabilize. Our main theorem, however, only shows that this stabilization happens for $n > 2k$ since any new vertex added by going from $n$ to $n+1$ is at distance at most $n/2$ from all existing vertices. In this section we prove the stronger statement, namely that the coefficients stabilize for $n > k$.
Let $$P_{C}:=\RI{\overline{\partial}C}\cap\bigcap_{v\in C}\BA{\{v\}}$$ be the event that every vertex in $C$ becomes active, and the boundary remains inactive. If $P_{C}$ holds, we say $C$ is a patch of the activations.

The intuition of the following theorem is similar to that of Corollary~\ref{cor:decay}. A site can only realize the length of the cycle after an interaction chain was formed around the cycle, implying that every vertex was activated at least once.

\begin{theorem} $R_{(n)}=\E_{[-m,m]}(\Res{0})+\bigO{p^{n}}$ for all $m\geq n \geq 3$, thus
	$R_{(n)}-R_{(m)}=\bigO{p^{n}}$.
\end{theorem}
\begin{proof} In the proof we identify the sites of the $n$-cycle with the$\mod n$ remainder classes.
    We have $R_{(n)}(p) = \E_{(n)}(\Res{0})$ by translation invariance, and this expectation is equal to $\sum_{k=1}^{\infty}\P_{(n)}(\Res{0}\!\geq\! k)$. We will now rewrite $X = \P_{(n)}(\Res{0}\!\geq\! k)$.
	\begin{align*}
        X
	&= \sum_{\underset{v+w\leq n+1}{v,w\in [n]}}\P_{(n)}(\Res{0}\!\geq\! k \cap \underset{P_{v,w}:=}{\underbrace{P_{[-v\!+\!1,w\!-\!1]}}}) \tag{partition}\\[-1mm]
	&= \sum_{\underset{v+w\leq n}{v,w\in [n]}} \P_{(n)}(\Res{0}\!\geq\! k \cap P_{v,w}) +\bigO{p^{n}}\\[-1mm]
	&= \sum_{\underset{v+w\leq n}{v,w\in [n]}} \P_{[-v,w]}(\Res{0}\!\geq\! k \cap P_{v,w} \mid \II{\{-v,w\}}) \P_{[w,n-v]}(\RI{w,n-v}) +\bigO{p^{n}} \tag{by Lemma~\ref{lemma:splitting}}\\
        &= \sum_{\smash{\underset{v+w\leq n}{v,w\in [n]}}}\P_{[-v,w]}(\Res{0}\!\geq\! k \cap P_{v,w} \mid \II{\{-v,w\}}) \\
    &\qquad\qquad\quad \cdot
        \left[ \left(\P_{[w,n-v]}(\RI{w})\right)^{\!\!2}\!+\!\bigO{p^{n-v-w+1}} \right] +\bigO{p^{n}}
        \tag{Corollary~\ref{cor:decay} and Equation~\eqref{eq:RICovariance}}\\
        &= \sum_{\smash{\underset{v+w\leq n}{v,w\in [n]}}} \P_{[-v,w]}(\Res{0}\!\geq\! k \cap P_{v,w} \mid \II{\{-v,w\}}) \\
    &\qquad\qquad\quad \cdot
        \left[\P_{[-m,-v]}(\RI{-v})\P_{[w,m]}(\RI{w})\!+\!\bigO{p^{n-v-w+1}} \right] +\bigO{p^{n}} \tag{by Lemma~\ref{lemma:distancePower}}\\
        &= \sum_{\smash{\underset{v+w\leq n}{v,w\in [n]}}} \P_{[-v,w]}(\Res{0}\!\geq\! k \cap P_{v,w} \mid \II{\{-v,w\}}) \P_{[-m,-v]}(\RI{-v})\P_{[w,m]}(\RI{w}) +\bigO{p^{n}} \tag{since $|P_{v,w}|=v+w-1$}\\
	&= \sum_{\underset{v+w\leq n}{v,w\in [n]}}\P_{[-m,m]}(\Res{0}\!\geq\! k \cap P_{v,w}) +\bigO{p^{n}} \tag{by Lemma~\ref{lemma:splitting}}\\[-1mm]
        &= \P_{[-m,m]}(\Res{0}\!\geq\! k) +\bigO{p^{n}} \tag{partition}
	\end{align*}  
    We conclude the proof by observing
    \begin{align*}
        \sum_{k=1}^\infty \P_{[-m,m]}(\Res{0}\!\geq\! k) +\bigO{p^{n}}= \E_{[-m,m]}(\Res{0})+\bigO{p^{n}} .
    \end{align*}
	\vskip-5mm
\end{proof}

\end{document}